\documentclass{amsart}
\usepackage{amsmath}
\usepackage{amssymb}
\usepackage{amsthm}

\DeclareMathOperator{\Id}{Id}

\DeclareMathOperator{\tr}{tr}

\DeclareMathOperator{\Vol}{Vol}
\DeclareMathOperator{\dvol}{dvol}
\DeclareMathOperator{\supp}{supp}

\DeclareMathOperator{\Ric}{Ric}

\DeclareMathOperator{\Met}{Met}

\newcommand{\omW}{\overline{\mathcal{W}}}

\newcommand{\lp}{\langle}
\newcommand{\rp}{\rangle}
\newcommand{\lv}{\lvert}
\newcommand{\rv}{\rvert}
\newcommand{\lV}{\lVert}
\newcommand{\rV}{\rVert}

\newcommand{\mQ}{\mathcal{Q}}

\newcommand{\mW}{\mathcal{W}}

\newcommand{\kM}{\mathfrak{M}}

\newcommand{\bN}{\mathbb{N}}

\newcommand{\bR}{\mathbb{R}}

\newcommand{\bZ}{\mathbb{Z}}

\newcommand{\comment}[1]{}

\newtheorem{thm}{Theorem}[section]
\newtheorem{prop}[thm]{Proposition}
\newtheorem{lem}[thm]{Lemma}
\newtheorem{cor}[thm]{Corollary}

\theoremstyle{definition}
\newtheorem{defn}[thm]{Definition}

\newtheorem{conj}[thm]{Conjecture}

\theoremstyle{remark}
\newtheorem{remark}[thm]{Remark}

\numberwithin{equation}{section}
\usepackage{esint}

\begin{document}

\title{A Yamabe-type problem on smooth metric measure spaces}
\author{Jeffrey S.\ Case}
\thanks{Partially supported by NSF Grant DMS-1004394}
\address{Department of Mathematics \\ Princeton University \\ Princeton, NJ 08540}
\email{jscase@math.princeton.edu}
\begin{abstract}
We describe and partially solve a natural Yamabe-type problem on smooth metric measure spaces which interpolates between the Yamabe problem and the problem of finding minimizers for Perelman's $\nu$-entropy.  In Euclidean space, this problem reduces to the characterization of the minimizers of the family of Gagliardo--Nirenberg inequalities studied by Del Pino and Dolbeault.  We show that minimizers always exist on a compact manifold provided the weighted Yamabe constant is strictly less than its value on Euclidean space.  We also show that strict inequality holds for a large class of smooth metric measure spaces, but we also give an example which shows that minimizers of the weighted Yamabe constant do not always exist.
\end{abstract}
\maketitle


\section{Introduction}
\label{sec:intro}

The Yamabe constant and Perelman's $\nu$-entropy are two important geometric invariants in Riemannian geometry which share many similarities.  Both constants are intimately related to sharp Sobolev-type inequalities on Euclidean space, with the Yamabe constant recovering the best constant for the Sobolev inequality and the $\nu$-entropy recovering the best constant for the logarithmic Sobolev inequality.  In the curved setting, these constants are defined as the infima of Sobolev-type quotients involving scalar curvature, and one can show that the infima are achieved by positive smooth functions through a two-step process.  First, one shows that minimizing sequences cannot concentrate provided the Yamabe constant (resp.\ $\nu$-entropy) is strictly less than the best constant for the Sobolev inequality (resp.\ logarithmic Sobolev inequality) on Euclidean space.  Second one shows that strict inequality always holds on a compact manifold, except in the case of the Yamabe constant on the standard conformal sphere.

It turns out that there is a natural one-parameter family of geometric invariants which interpolate between the Yamabe constant and the $\nu$-entropy.  These invariants, which we call \emph{weighted Yamabe constants}, were introduced by the author~\cite{Case2011gns} as curved analogues of the best constants in the family of Gagliardo--Nirenberg inequalities studied by Del Pino and Dolbeault~\cite{DelPinoDolbeault2002}.  The purpose of this article is to study to what extent these invariants interpolate between the Yamabe constant and the $\nu$-entropy, focusing on issues related to the problem of finding minimizers of the weighted Yamabe quotients.

In order to explain our results, we first recall the aforementioned result of Del Pino and Dolbeault~\cite{DelPinoDolbeault2002}.

\begin{thm}[Del Pino--Dolbeault]
\label{thm:dd}
Fix $m\in[0,\infty)$.  Given any $w\in W^{1,2}(\bR^n)\cap L^{\frac{2(m+n)}{m+n-2}}(\bR^n)$ it holds that
\begin{equation}
\label{eqn:dd}
\Lambda_{m,n}\left( \int_{\bR^n} w^{\frac{2(m+n)}{m+n-2}}\right)^{\frac{2m+n-2}{n}} \leq \left(\int_{\bR^n} \lv\nabla w\rv^2\right) \left(\int_{\bR^n} w^{\frac{2(m+n-1)}{m+n-2}}\right)^{\frac{2m}{n}},
\end{equation}
where the constant $\Lambda_{m,n}$ is given by
\begin{equation}
\label{eqn:dd_yamabe}
\Lambda_{m,n} = \frac{n\pi(m+n-2)^2}{2m+n-2}\left(\frac{2(m+n-1)}{2m+n-2}\right)^{\frac{2m}{n}}\left(\frac{\Gamma\left(\frac{2m+n}{2}\right)}{\Gamma(m+n)}\right)^{\frac{2}{n}} .
\end{equation}
Moreover, equality holds in~\eqref{eqn:dd} if and only if there is a constant $\varepsilon>0$ and a point $x_0\in\bR^n$ such that $w$ is a constant multiple of the function
\begin{equation}
\label{eqn:dd_bubble}
w_{\varepsilon,x_0}(x) := \left(\frac{2\varepsilon}{\varepsilon^2+\lv x-x_0\rv^2}\right)^{\frac{m+n-2}{2}} .
\end{equation}
\end{thm}

There are four features of Theorem~\ref{thm:dd} which we wish to emphasize.  First, Theorem~\ref{thm:dd} recovers the sharp Sobolev inequality~\cite{Aubin1976s,Talenti1976} in the case $m=0$ and the sharp logarithmic Sobolev inequality in the case $m=\infty$.  Second, the extremal functions~\eqref{eqn:dd_bubble} are all the same, except for the dependence of the exponent on the parameter $m$.  Third, the functions $w_{\varepsilon,x_0}$ concentrate at $x_0$ as $\varepsilon\to0$.  Fourth, the family~\eqref{eqn:dd} of Gagliardo--Nirenberg (GN) inequalities is, in a certain sense, the only such family with geometrically significant extremal functions.  This last point requires further explanation.

Given constants $2\leq p\leq q\leq\frac{2n}{n-2}$, the sharp Sobolev inequality and H\"older's inequality yield a positive constant $C_{p,q}$ such that the GN inequality
\begin{equation}
\label{eqn:gns}
\lV w\rV_q \leq C_{p,q}\lV\nabla w\rV_2^\theta \lV w\rV_p^{1-\theta}
\end{equation}
holds for all $w\in C_0^\infty(\bR^n)$.  At present, only in the case $2p=q+2$, corresponding to the family~\eqref{eqn:dd}, is the best constant $C_{p,q}$ known (there are other cases known in the range $1\leq p\leq q\leq\frac{2n}{n-2}$; e.g.\ \cite{CarlenLoss1993,DelPinoDolbeault2002}).  This leads one to wonder if there is some geometric property distinguishing this family.  One such property was previously described by the author~\cite{Case2011gns}: The formalism of smooth metric measure spaces allows one to define conformal invariants which give a curved analogue of the sharp constant $C_{p,q}$ in~\eqref{eqn:gns} as the infimum of the total weighted scalar curvature subject to certain volume constraints.  In this framework, the family~\eqref{eqn:dd} has the property that it is the only family of GN inequalities~\eqref{eqn:gns} for which the extremal functions on Euclidean space are also critical points of the constrained total weighted scalar curvature functional through variations of the metric or the measure.  This generalizes the fact that extremal functions of the Sobolev inequality (resp.\ logarithmic Sobolev inequality) give rise to conformally flat Einstein metrics on $\bR^n$ (resp.\ Gaussian measures on $\bR^n)$.

To explain the results of this article requires some terminology.  A \emph{smooth metric measure space} is a four-tuple $(M^n,g,e^{-\phi}\dvol,m)$ of a Riemannian manifold $(M^n,g)$, a smooth measure $e^{-\phi}\dvol$ determined by a function $\phi\in C^\infty(M)$ and the Riemannian volume element of $g$, and a dimensional parameter $m\in[0,\infty]$.  The \emph{weighted scalar curvature $R_\phi^m$} of a smooth metric measure space is $R_\phi^m:=R+2\Delta\phi-\frac{m+1}{m}\lv\nabla\phi\rv^2$, where $R$ and $\Delta$ are the scalar curvature and Laplacian associated to the metric $g$, respectively.  The \emph{weighted Yamabe quotient} is the functional
\begin{subequations}
\label{eqn:intro_mQ}
\begin{equation}
\label{eqn:intro_mQ_finite}
\mQ(w) := \frac{\left(\int \lv\nabla w\rv^2+\frac{m+n-2}{4(m+n-1)}R_\phi^mw^2\right)\left(\int \lv w\rv^{\frac{2(m+n-1)}{m+n-2}}e^{\phi/m}\right)^{\frac{2m}{n}}}{\left(\int \lv w\rv^{\frac{2(m+n)}{m+n-2}}\right)^{\frac{2m+n-2}{n}}} ,
\end{equation}
where all integrals are taken with respect to $e^{-\phi}\dvol$; in the limit $m=\infty$, this is
\begin{equation}
\label{eqn:intro_mQ_infinite}
\mQ(w) := \frac{\int \lv\nabla w\rv^2+\frac{1}{4}R_\phi^\infty w^2}{\int w^2}\exp\left(-\frac{2}{n}\int_M \frac{w^2}{\lV w\rV_2^2}\log\frac{w^2e^{-\phi}}{\lV w\rV^2}\right) .
\end{equation}
\end{subequations}
The weighted Yamabe quotient is conformally invariant in the sense that if
\[ \left( M^n, \hat g, e^{-\hat\phi}\dvol_{\hat g}, m \right) = \left( M^n, e^{\frac{2\sigma}{m+n-2}}g, e^{\frac{(m+n)\sigma}{m+n-2}}e^{-\phi}\dvol_g\right) \]
for some $\sigma\in C^\infty(M)$, then $\hat\mQ(w)=\mQ(we^{\sigma/2})$.  There are similar conformally invariant functionals on smooth metric measure spaces generalizing~\eqref{eqn:gns} for $2\leq p\leq\frac{2(m+n)}{m+n-2}=q$, and it is through these functionals that one obtains the characterization described in the previous paragraph.

The \emph{weighted Yamabe constant} of a compact smooth metric measure space is
\[ \Lambda[g,e^{-\phi}\dvol,m] := \inf\left\{ \mQ(w) \colon 0<w\in C^\infty(M)\right\} . \]
When $m=0$, this is the Yamabe constant.  When $m=\infty$ and $\Lambda>0$, this is equivalent to Perelman's $\nu$-entropy~\cite{Perelman1}; see Section~\ref{sec:bg} for details.  Thus the weighted Yamabe constant interpolates between the Yamabe constant and Perelman's $\nu$-entropy.  In this paper we study the \emph{weighted Yamabe problem}, which asks about the existence of functions which minimize the weighted Yamabe quotient.  We also consider the uniqueness of these functions in a geometrically significant setting.  Our results illustrate the interpolatory nature of the weighted Yamabe constants, though, as we describe below, there are some surprises.

Our approach to these problems is similar to approaches to the Yamabe Problem~\cite{Aubin1976,Obata1971,Schoen1984,Trudinger1968,Yamabe1960} and to Perelman's $\nu$-entropy~\cite{Perelman1}.  Much of the analysis is based on the Euler--Lagrange equation
\begin{equation}
\label{eqn:intro/euler}
-\Delta_\phi w + \frac{m+n-2}{4(m+n-1)}R_\phi^mw + c_1w^{\frac{m+n}{m+n-2}}e^{\frac{\phi}{m}} = c_2w^{\frac{m+n+2}{m+n-2}}
\end{equation}
for critical points of the functional $\mQ$.  When $m>0$, the equation~\eqref{eqn:intro/euler} has a subcritical nonlinearity.  The main difficulty is instead that minimizing sequences for the weighted Yamabe constant need not be uniformly bounded in $W^{1,2}(M)$.  We overcome this difficulty by introducing a generalization of Perelman's $\mW$-functional.  Using this functional, we obtain an Aubin-type criterion for the existence of minimizers of the weighted Yamabe constant.

\begin{thm}
\label{thm:blow_up}
Let $(M^n,g,e^{-\phi}\dvol,m)$ be a compact smooth metric measure space.  Then
\begin{equation}
\label{eqn:weighted_yamabe_estimate}
\Lambda[g,e^{-\phi}\dvol,m] \leq \Lambda[\bR^n,dx^2,\dvol,m] = \Lambda_{m,n}.
\end{equation}
Moreover, if the inequality~\eqref{eqn:weighted_yamabe_estimate} is strict, then there exists a positive function $w\in C^\infty(M)$ such that
\[ \mQ(w) = \Lambda[g,e^{-\phi}\dvol,m] . \]
\end{thm}

That the weighted Yamabe constant of Euclidean space $(\bR^n,dx^2,\dvol,m)$ is $\Lambda_{m,n}$ follows from Theorem~\ref{thm:dd}.

We can solve the weighted Yamabe problem when $m\in\bN\cup\{0,\infty\}$ using the following necessary condition for equality to hold in~\eqref{eqn:weighted_yamabe_estimate}.  Moreover, if Conjecture~\ref{conj:weighted_obata} below is true, then this is also a sufficient condition.

\begin{thm}
\label{thm:integral_characterization}
Let $(M^n,g,e^{-\phi}\dvol,m)$ be a compact smooth metric measure space such that $m\in\bN\cup\{0,\infty\}$.  If
\[ \Lambda[g,e^{-\phi}\dvol,m] = \Lambda[\bR^n,dx^2,\dvol,m], \]
then $m\in\{0,1\}$ and $(M^n,g,e^{-\phi}\dvol,m)$ is conformally equivalent to $(S^n,g_0,\dvol,m)$ for $g_0$ a metric of constant sectional curvature.  In particular, there exists a positive function $w\in C^\infty(M)$ such that
\[ \mQ(w) = \Lambda\left[g,e^{-\phi}\dvol,m\right] . \]
\end{thm}

One reason we have been unable to give satisfactory necessary conditions for equality to hold in~\eqref{eqn:weighted_yamabe_estimate} for all $m\in[0,\infty]$, and in particular solve the weighted Yamabe problem, is that minimizers of the weighted Yamabe constant do not always exist.  As a consequence, neither the Aubin--Schoen argument~\cite{Aubin1976,LeeParker1987,Schoen1984} nor the Perelman argument~\cite{Perelman1} characterizing equality in~\eqref{eqn:weighted_yamabe_estimate} for $m=0$ or $m=\infty$, respectively, can be generalized to all $m\in[0,\infty]$.  To prove Theorem~\ref{thm:integral_characterization}, we instead use minimizers of the weighted Yamabe constant of $(M^n,g,\dvol,m)$ as test functions to estimate the weighted Yamabe constant of $(M^n,g,\dvol,m+1)$; see Theorem~\ref{thm:constants_reln} for details.  In particular, this allows us to iterate the Aubin--Schoen characterization to all integers $m$.

That minimizers of the weighted Yamabe constant do not always exist is a consequence of the following surprising result.

\begin{thm}
\label{thm:nonexistence}
Minimizers for the weighted Yamabe constant of $(S^n,g_0,1^{1/2}\dvol)$ do not exist.
\end{thm}

The proof of Theorem~\ref{thm:nonexistence} is based on an estimate relating the weighted Yamabe constant of $(M^n,g,e^{-\phi}\dvol,m)$ to the Yamabe constant of $(M^n\times\bR^{2m},g\oplus e^{-2\phi/m}dy^2)$ when $2m\in\bN$.  This relationship is the curved generalization of an observation of Bakry (cf.\ \cite{BakryGentilLedoux2012,CarlenFigalli2011}) used to give an alternative proof of Theorem~\ref{thm:dd}.

We expect that the weighted Yamabe problem is always solvable for $m\in\{0\}\cup[1,\infty)$, but not for $m\in(0,1)$.  Assuming Conjecture~\ref{conj:weighted_obata} below holds, we show that positive smooth minimizers of the weighted Yamabe constant of $(S^n,g_0,\dvol,m)$ do not exist for any $m\in(0,1)$.  On the other hand, we show that for $m>1$, the weighted Yamabe constant of $(S^n,g_0,\dvol,m)$ is strictly less than $\Lambda_{m,n}$.

Critical points of the weighted Yamabe quotient for $m=0$ and $g$ Einstein (resp.\ $m=\infty$ and $g$ a shrinking gradient Ricci soliton) have been characterized by Obata~\cite{Obata1971} (resp.\ Perelman~\cite{Perelman1}).  Such metrics are critical points for $\mQ$ through variations of the metric and the measure.  We conjecture that the corresponding result holds for all $m$.

\begin{conj}
\label{conj:weighted_obata}
Let $(M^n,g,e^{-\phi}\dvol,m)$ be a compact smooth metric measure space and suppose that there exists a constant $\lambda\in\bR$ such that
\begin{equation}
\label{eqn:weighted_einstein}
\Ric_\phi^m - \frac{R_\phi^m}{2(m+n-1)}g = \frac{m+n-2}{2(m+n-1)}\lambda g,
\end{equation}
where $\Ric_\phi^m$ is the Bakry-\'Emery Ricci tensor.  If $w\in C^\infty(M)$ is a positive critical point of the map $w\mapsto\mQ(w)$, then either
\begin{enumerate}
\item $w$ is constant, or
\item $m\in\{0,1\}$,
\[ \left(M^n,\hat g,e^{-\hat\phi}\dvol_{\hat g},m\right) := \left( M^n, w^{\frac{4}{m+n-2}}g, w^{\frac{2(m+n)}{m+n-2}}e^{-\phi}\dvol_g,m \right) \]
satisfies~\eqref{eqn:weighted_einstein} for some constant $\hat\lambda>0$, and both $(M^n,g)$ and $(M^n,\hat g)$ are homothetic to the standard $n$-sphere.
\end{enumerate}
\end{conj}

Unfortunately, the obvious modification of the proofs of Conjecture~\ref{conj:weighted_obata} in the cases $m=0$~\cite{Obata1971} and $m=\infty$~\cite{Perelman1} does not seem to work for general $m$.  The difficulty comes from the uncertainty of the signs of certain lower order terms which appear.  Nevertheless, the Obata--Perelman argument does prove Conjecture~\ref{conj:weighted_obata} under one of the following two additional assumptions:
\begin{enumerate}
\item $(w,\phi)$ is a critical point of the map $(\xi,\psi)\mapsto\mQ[g,e^{-\psi}\dvol,m](\xi)$; or
\item $(M^n,g,e^{-\phi}\dvol,m)$ is isometric to Euclidean space and the function
\begin{equation}
\label{eqn:intro/obata_decay}
\tilde w(x) := \lv x\rv^{2-m-n}w\left(\frac{x}{\lv x\rv^2}\right)
\end{equation}
can be extended to a positive function in $C^2(\bR^n)$.
\end{enumerate}
For precise statements, see Proposition~\ref{prop:weighted_obata_measure_critical} and Proposition~\ref{prop:weighted_obata_euclidean}, respectively.  In particular, the second assumption nearly yields an alternative proof of the characterization of the extremal functions~\eqref{eqn:dd_bubble} in Theorem~\ref{thm:dd}.  Since Conjecture~\ref{conj:weighted_obata} is more interesting as a statement about curved smooth metric measure spaces, we do not here try to remove the assumption~\eqref{eqn:intro/obata_decay}.

This article is organized as follows.

In Section~\ref{sec:smms} we give a more detailed introduction to smooth metric measure spaces and the ways in which they will be studied in this article.

In Section~\ref{sec:bg} we discuss basic properties of the weighted Yamabe quotient and introduce our generalization of Perelman's $\mW$-functional~\cite{Perelman1}.

In Section~\ref{sec:variation} we compute the first variations of the weighted Yamabe quotient and our $\mW$-functional.  This clarifies the role of the assumption~\eqref{eqn:weighted_einstein} in Conjecture~\ref{conj:weighted_obata}.

In Section~\ref{sec:model} we give a geometric description of Theorem~\ref{thm:dd} as the solution of the weighted Yamabe problem on Euclidean space.  We also establish estimates necessary to study concentration for minimizing sequences of the weighted Yamabe constant of arbitrary compact smooth metric measure spaces.

In Section~\ref{sec:existence} we study the analytic aspects of the weighted Yamabe problem, culminating in the proof of Theorem~\ref{thm:blow_up}.

In Section~\ref{sec:yamabe} we establish the relationship between the weighted Yamabe constants of $(M^n,g,\dvol,m)$ and $(M^n,g,\dvol,m+1)$, and then prove Theorem~\ref{thm:integral_characterization}.

In Section~\ref{sec:product} we establish the relationship between the weighted Yamabe constant of $(M^n,g,e^{-\phi}\dvol,m)$ and the Yamabe constant of $(M\times\bR^{2m},g\oplus e^{-2\phi/m} dx^2)$, and then prove Theorem~\ref{thm:nonexistence}.

In Section~\ref{sec:uniqueness} we present our uniqueness results.  We also discuss the weighted Yamabe constants of $(S^n,g_0,\dvol,m)$ for all $m\in[0,\infty]$.


\section{Smooth metric measure spaces}
\label{sec:smms}

We collect in this section some basic definitions and facts for smooth metric measure spaces as will be needed in this article (cf.\ \cite{Case2010a}).

\begin{defn}
A \emph{smooth metric measure space} is a four-tuple $(M^n,g,e^{-\phi}\dvol,m)$ of a Riemannian manifold $(M^n,g)$, a smooth measure $e^{-\phi}\dvol$ determined by $\phi\in C^\infty(M)$ and the Riemannian volume element $\dvol$ of $g$, and a dimensional parameter $m\in[0,\infty]$.  In the case $m=0$, we require $\phi=0$.
\end{defn}

We frequently denote smooth metric measure spaces as triples $(M^n,g,v^m\dvol)$, where the measure is $v^m\dvol$ and the dimensional parameter is encoded as the exponent of $v$.  In accordance with this convention, $v$ and $\phi$ will denote throughout this article functions which are related by $v^m=e^{-\phi}$; when $m=\infty$, this is to be interpreted as the formal definition of the symbol $v^\infty$.

\begin{defn}
The \emph{weighted divergence $\delta_\phi$} of a smooth metric measure space $(M^n,g,v^m\dvol)$ is the operator defined on tensor fields $T$ by
\[ \left(\delta_\phi T\right)(X,Y,\dotsc) = \sum_{i=1}^n \nabla_{E_i} T(E_i,X,Y,\dotsc) - T(\nabla\phi,X,Y,\dotsc) \]
for all $p\in M$ and all vector fields $X,Y$ in a neighborhood of $p$, where $\{E_i\}$ is a parallel orthonormal frame in a neighborhood of $p$ and we use the metric $g$ to change the type of $T$ if necessary.
\end{defn}

\begin{defn}
Let $(M^n,g,v^m\dvol)$ be a smooth metric measure space.  The \emph{weighted Laplacian $\Delta_\phi\colon C^\infty(M)\to C^\infty(M)$} is the operator
\[ \Delta_\phi = \delta_\phi d = \Delta - \nabla\phi . \]
\end{defn}

\begin{defn}
\label{defn:weighted_curvature}
Let $(M^n,g,v^m\dvol)$ be a smooth metric measure space.  The \emph{Bakry-\'Emery Ricci curvature $\Ric_\phi^m$} and the \emph{weighted scalar curvature $R_\phi^m$} are the tensors
\begin{align*}
\Ric_\phi^m & := \Ric + \nabla^2\phi - \frac{1}{m}d\phi\otimes d\phi \\
R_\phi^m & := R + 2\Delta\phi - \frac{m+1}{m}\lv\nabla\phi\rv^2 .
\end{align*}
\end{defn}

\begin{defn}
\label{defn:scms}
Two smooth metric measure spaces $(M^n,g,e^{-\phi}\dvol_g,m)$ and $(M^n,\hat g,e^{-\hat\phi}\dvol_{\hat g},m)$ are \emph{pointwise conformally equivalent} if there is a function $\sigma\in C^\infty(M)$ such that
\begin{equation}
\label{eqn:scms}
\left( M^n, \hat g, e^{-\hat\phi}\dvol_{\hat g}, m\right) = \left( M^n, e^{\frac{2}{m+n-2}\sigma}g, e^{\frac{m+n}{m+n-2}\sigma}e^{-\phi}\dvol_g, m\right) .
\end{equation}
$(M^n,g,e^{-\phi}\dvol_g,m)$ and $(\hat M^n,\hat g,e^{-\hat\phi}\dvol_{\hat g},m)$ are \emph{conformally equivalent} if there is a diffeomorphism $F\colon\hat M\to M$ such that $\bigl(\hat M^n,F^\ast g,F^\ast(e^{-\phi}\dvol_{g}),m\bigr)$ is pointwise conformally equivalent to $(\hat M^n,\hat g,e^{-\hat\phi}\dvol_{\hat g},m)$.
\end{defn}

\begin{defn}
\label{defn:conformal_invariant}
A geometric invariant $T\left[g,e^{-\phi}\dvol,m\right]$ of a smooth metric measure space $(M^n,g,e^{-\phi}\dvol,m)$ is \emph{conformally invariant} if for all $\sigma\in C^\infty(M)$,
\[ T\left[e^{\frac{2}{m+n-2}\sigma}g,e^{\frac{m+n}{m+n-2}\sigma}e^{-\phi}\dvol_g,m\right] = T\left[g,e^{-\phi}\dvol_g,m\right] . \]
\end{defn}

\begin{defn}
\label{defn:conformally_covariant}
An operator $T[g,e^{-\phi}\dvol,m]\colon C^\infty(M)\to C^\infty(M)$ on a smooth metric measure space $(M^n,g,e^{-\phi}\dvol,m)$ is \emph{conformally covariant of bidegree $(a,b)$} if for all $\sigma\in C^\infty(M)$,
\[ T\left[e^{\frac{2}{m+n-2}\sigma}g,e^{\frac{m+n}{m+n-2}\sigma}e^{-\phi}\dvol_g,m\right] = e^{-\frac{(m+n)b}{m+n-2}\sigma}\circ T\left[g,e^{-\phi}\dvol_g,m\right]\circ e^{\frac{(m+n)a}{m+n-2}\sigma}, \]
where the right hand side denotes a pre- and post-composition of $T$ with two multiplication operators.
\end{defn}

The simplest nontrivial example of a conformally covariant operator, and the one which we study in this article, is the weighted conformal Laplacian.

\begin{defn}
Let $(M^n,g,v^m\dvol)$ be a smooth metric measure space.  The \emph{weighted conformal Laplacian $L_\phi^m\colon C^\infty(M)\to C^\infty(M)$} is the operator
\[ L_\phi^m := -\Delta_\phi + \frac{m+n-2}{4(m+n-1)}R_\phi^m . \]
\end{defn}

\begin{prop}[{\cite[Lemma~3.4]{Case2010b}}]
\label{prop:weighted_conformal_laplacian_invariant}
The weighted conformal Laplacian is a conformally covariant operator of bidegree $\left(\frac{m+n-2}{2(m+n)},\frac{m+n+2}{2(m+n)}\right)$.
\end{prop}
\section{The weighted Yamabe constant}
\label{sec:bg}

In this section we define the weighted Yamabe quotient, the $\mW$-functional, and their associated energies and describe some of their basic properties.

\subsection{The weighted Yamabe quotient}

\begin{defn}
Let $(M^n,g,v^m\dvol)$ be a compact smooth metric measure space.  The \emph{weighted Yamabe quotient $\mQ\colon C^\infty(M)\to\bR$} is defined by
\begin{subequations}
\label{eqn:conformal_gns_quotient}
\begin{equation}
\label{eqn:conformal_gns_quotient_finite}
\mQ[g,v^m\dvol](w) = \frac{ (L_\phi^m w,w)\,\left(\int \lv w\rv^{\frac{2(m+n-1)}{m+n-2}}v^{-1}\right)^{\frac{2m}{n}} }{ \left(\int \lv w\rv^{\frac{2(m+n)}{m+n-2}}\right)^{\frac{2m+n-2}{n}}}
\end{equation}
when $m<\infty$ and by
\begin{equation}
\label{eqn:conformal_gns_quotient_infty}
\mQ[g,e^{-\phi}\dvol](w) = \frac{(L_\phi^\infty w,w)}{\lV w\rV_2^2} \exp\left( -\frac{2}{n}\int_M\frac{w^2}{\lV w\rV_2^2}\log\frac{w^2e^{-\phi}}{\lV w\rV_2^2} \right)
\end{equation}
\end{subequations}
when $m=\infty$.

The \emph{weighted Yamabe constant} $\Lambda[g,v^m\dvol]\in\bR$ of $(M^n,g,v^m\dvol)$ is
\[ \Lambda[g,v^m\dvol] = \inf\left\{ \mQ[g,v^m\dvol](w) \colon 0\not=w\in W^{1,2}(M,v^m\dvol) \right\} . \]
\end{defn}

As usual in this article, all integrals computed using the measure $v^m\dvol$.  Here $W^{1,2}(M,v^m\dvol)$ denotes the closure of $C^\infty(M)$ with respect to the norm
\[ \lV w\rV_{W^{1,2}(M,v^m\dvol)} := \int_M\left(\lv\nabla w\rv^2+w^2\right) . \]
Since $M$ is compact and $v$ is positive, this norm is equivalent to the usual $W^{1,2}$-norm.  Since $C^\infty(M)$ is dense in $W^{1,2}(M)$ and $\mQ(\lv w\rv)\leq\mQ(w)$, we may equivalently define the weighted Yamabe constant by minimizing over the space of positive smooth functions on $M$, as we shall often do without further comment.

The weighted Yamabe quotient is continuous in $m\in[0,\infty]$.

\begin{prop}
\label{prop:yamabe_quotient_limit}
Let $(M^n,g)$ be a compact Riemannian manifold and fix $\phi\in C^\infty(M)$ and $m\in[0,\infty]$.  Given any $w\in C^\infty(M)$, it holds that
\[ \lim_{k\to m} \mQ[g,e^{-\phi}\dvol,k](w) = \mQ[g,e^{-\phi}\dvol,m](w) . \]
\end{prop}

\begin{proof}

This is clear in the case $m<\infty$, while in the case $m=\infty$ it follows easily from the expansion
\[ \frac{\int \lv w\rv^{\frac{2(m+n-1)}{m+n-2}}e^{-\frac{m-1}{m}\phi}\dvol}{\int \lv w\rv^{\frac{2(m+n)}{m+n-2}}e^{-\phi}\dvol} = 1 - \frac{1}{m}\left(\int_M \frac{w^2}{\lV w\rV_2^2}\log(w^2e^{-\phi})e^{-\phi}\dvol\right) + O(m^{-2}) \]
for $m$ large (cf.\ \cite{DelPinoDolbeault2002}).
\end{proof}

It is easily checked that the weighted Yamabe quotient is invariant under separate constant rescalings of the metric, the measure, and the function $w$.  Moreover, it is conformally covariant.

\begin{prop}
\label{prop:yamabe_quotient_invariant}
Let $(M^n,g,v^m\dvol)$ be a compact smooth metric measure space.  For any $\sigma,w\in C^\infty(M)$ it holds that
\[ \mQ\left[e^{\frac{2}{m+n-2}\sigma}g,e^{\frac{m+n}{m+n-2}\sigma}v^m\dvol_g\right](w) = \mQ\left[g,v^m\dvol_g\right]\left(e^{\frac{1}{2}\sigma}w\right) . \]
In particular,
\[ \Lambda\left[e^{\frac{2}{m+n-2}\sigma}g,e^{\frac{m+n}{m+n-2}\sigma}v^m\dvol_g\right] = \Lambda\left[g,v^m\dvol\right] . \]
\end{prop}

\begin{proof}

It is clear that the integrals
\[ \int_M \lv w\rv^{\frac{2(m+n-1)}{m+n-2}}v^{m-1}\dvol \quad\text{and}\quad \int_M \lv w\rv^{\frac{2(m+n)}{m+n-2}}\,v^m\dvol \]
are invariant under the conformal transformation
\begin{equation}
\label{eqn:conformal_transformation}
\left( g,v^m\dvol,w\right) \mapsto \left( e^{\frac{2}{m+n-2}\sigma}g,e^{\frac{m+n}{m+n-2}\sigma}v^m\dvol_g,e^{-\frac{1}{2}\sigma}w \right) .
\end{equation}
That $(L_\phi^mw,w)$ is invariant under~\eqref{eqn:conformal_transformation} follows from Proposition~\ref{prop:weighted_conformal_laplacian_invariant}.
\end{proof}

Note that $\int w^{\frac{2(m+n)}{m+n-2}}$ measures the weighted volume $\int e^{-\hat\phi}\dvol_{\hat g}$ of
\[ \left( M^n, \hat g, \hat v^m\dvol_{\hat g} \right) := \left( M^n, w^{\frac{4}{m+n-2}}g, w^{\frac{2(m+n)}{m+n-2}}v^m\dvol_g \right) . \]
In order to remove the trivial source of noncompactness in the weighted Yamabe problem, we typically make the following normalization for $w$.

\begin{defn}
Let $(M^n,g,v^m\dvol)$ be a smooth metric measure space.  We say that a positive function $w\in C^\infty(M)$ is \emph{volume-normalized} if
\begin{equation}
\label{eqn:volume_normalized}
\int_M w^{\frac{2(m+n)}{m+n-2}}v^m\dvol = 1 .
\end{equation}
\end{defn}

We conclude this subsection with two useful observations.  First, the sign of the weighted Yamabe constant is the same as the sign of the weighted conformal Laplacian.

\begin{prop}
\label{prop:signs}
Let $(M^n,g,v^m\dvol)$ be a compact smooth metric measure space and denote
\[ \lambda_1(L_\phi^m) := \inf \left\{ \frac{(L_\phi^mw,w)}{\lV w\rV_2^2} \colon 0\not=w\in W^{1,2}(M,v^m\dvol) \right\} . \]
Then exactly one of the three following statements is true:
\begin{enumerate}
\item $\lambda_1(L_\phi^m)$ and $\Lambda[g,v^m\dvol]$ are both positive.
\item $\lambda_1(L_\phi^m)$ and $\Lambda[g,v^m\dvol]$ are both zero.
\item $\lambda_1(L_\phi^m)$ and $\Lambda[g,v^m\dvol]$ are both negative.
\end{enumerate}
\end{prop}

\begin{proof}

Denote $\lambda_1=\lambda_1(L_\phi^m)$ and $\Lambda=\Lambda[g,v^m\dvol]$.  It is clear that $\lambda_1<0$ if and only if $\Lambda<0$.  A standard argument in elliptic PDE shows that there is a positive function $w\in C^\infty(M)$ such that $L_\phi^mw=\lambda_1w$.  Hence, if $\lambda_1=0$, then $\Lambda=0$.  If instead $\lambda>0$, then Proposition~\ref{prop:weighted_conformal_laplacian_invariant} implies that $\hat g:=w^{\frac{4}{m+n-2}}g$ and $\hat v:=w^{\frac{2}{m+n-2}}v$ are such that $(M^n,\hat g,\hat v^m\dvol)$ has strictly positive weighted scalar curvature.  The Sobolev inequality (cf.\ \cite{Hebey1999}) for $(M^n,\hat g)$ then yields a uniform bound
\[ \mQ[\hat g,\hat v^m\dvol](w) \geq C > 0 \]
for all $w\in C^\infty(M)$.  Proposition~\ref{prop:yamabe_quotient_invariant} thus implies that $\Lambda>0$.
\end{proof}

Second, the sign of the weighted Yamabe constant is monotone in $m$.

\begin{prop}
\label{prop:monotone_in_m}
Let $(M^n,g,v^m\dvol)$ be a compact smooth metric measure space with negative (resp.\ nonpositive) weighted Yamabe constant.  Then for any $k\geq0$, the weighted Yamabe constant of $(M^n,g,v^{m+k}\dvol)$ is negative (resp.\ nonpositive).
\end{prop}

\begin{remark}
As a corollary, we have that if the weighted Yamabe constant is positive (resp.\ nonnegative), then so too is the Yamabe constant of the underlying Riemannian manifold.  Hence the weighted Yamabe constant has topological implications (cf.\ \cite{GromovLawson1983}).
\end{remark}

\begin{proof}

It suffices to show that if the first eigenvalue of the weighted conformal Laplacian $L_\phi^m$ of $(M^n,g,1^m\dvol)$ is negative (resp.\ nonpositive), then so too is the first eigenvalue of the weighted conformal Laplacian $L_\phi^{m+k}$ of $(M^n,g,1^{m+k}\dvol)$.  Indeed, Proposition~\ref{prop:yamabe_quotient_invariant} implies that
\begin{align*}
\Lambda[g,v^{m+k}\dvol_g] & = \Lambda[v^{-2}g,1^{m+k}\dvol_{v^{-2}g}], \\
\Lambda[g,v^m\dvol_g] & = \Lambda[v^{-2}g,1^m\dvol_{v^{-2}g}] ,
\end{align*}
while Proposition~\ref{prop:signs} allows us to consider instead the first eigenvalue of the weighted conformal Laplacian.  A straightforward computation shows that
\begin{equation}
\label{eqn:wcl_relationship}
\frac{(m+k+n-1)(m+n-2)}{(m+k+n-2)(m+n-1)}\left( L_\phi^{m+k}w,w\right) \leq \left( L_\phi^mw,w\right)
\end{equation}
for all $w\in W^{1,2}(M)$, with both sides computed with respect to the respective smooth metric measure space.  Let $w\in C^\infty(M)$ satisfy $L_\phi^mw=\lambda_1(L_\phi^m)w$.  Inserting this function into~\eqref{eqn:wcl_relationship} yields then conclusion.
\end{proof}

\subsection{The $\mW$-functional}

One difficulty which arises when trying to minimize the weighted Yamabe quotient directly is that it cannot be used to show that minimizing sequences are \emph{a priori} bounded in $W^{1,2}(M,v^m\dvol)$.  More precisely, there is no reason that a minimizing sequence $\{w_i\}$ of volume-normalized functions must also have $(L_\phi^mw_i,w_i)$ uniformly bounded above.  Indeed, the explicit minimizers~\eqref{eqn:dd_bubble} in Theorem~\ref{thm:dd} exhibit this behavior as $\varepsilon\to0$.

To overcome this difficulty, we recast the weighted Yamabe problem as an optimization problem in $w$ and an additional scale parameter which reflects the energy $(L_\phi^mw,w)$.  This is accomplished through the introduction of the following generalization of Perelman's $\mW$-functional~\cite{Perelman1}, and parallels the approach of Del Pino and Dolbeault~\cite{DelPinoDolbeault2002} to proving the existence of extremal functions for~\eqref{eqn:dd}.

\begin{defn}
Let $(M^n,g,v^m\dvol)$ be a compact smooth metric measure space.  The \emph{$\mW$-functional} $\mW\colon C^\infty(M)\times\bR^+\to\bR$ is defined by
\begin{subequations}
\label{eqn:mw}
\begin{equation}
\label{eqn:mw_finite}
\begin{split}
\mW\left(w,\tau\right) & = \tau^{\frac{m}{m+n}}\left(L_\phi^mw,w\right) + m\int_M \left(\tau^{-\frac{n}{2(m+n)}}w^{\frac{2(m+n-1)}{m+n-2}}v^{-1} - w^{\frac{2(m+n)}{m+n-2}}\right)
\end{split}
\end{equation}
when $m<\infty$ and by
\begin{equation}
\label{eqn:mw_infty}
\mW\left(w,\tau\right) = \tau\left(L_\phi^\infty w,w\right) - \int_M w^2\log\left(\tau^{\frac{n}{2}}w^2e^{-\phi}\right)
\end{equation}
\end{subequations}
when $m=\infty$.
\end{defn}

When $m=\infty$, the $\mW$-functional is, up to dimensional constants, Perelman's $\mW$-functional; see below for details.  The $\mW$-functional allows us to better understand the issue of concentration of minimizing sequences of the weighted Yamabe quotient when the weighted Yamabe constant is positive.

Similar to Proposition~\ref{prop:yamabe_quotient_limit}, the $\mW$-functional is continuous for $m\in[0,\infty]$.

\begin{lem}
\label{lem:gns_functional_limit}
Let $(M^n,g)$ be a compact Riemannian manifold and fix $\phi\in C^\infty(M)$ and $m\in[0,\infty]$.  Given any $w\in C^\infty(M)$, it holds that
\[ \lim_{k\to m} \mW\left[g,e^{-\phi}\dvol,k\right](w) = \mW\left[g,e^{-\phi}\dvol,m\right](w) . \]
\end{lem}

The proof is a simple calculus exercise, and will be omitted.

One way to regard the scale parameter $\tau$ is as a mechanism to break the freedom to rescale the measure $v^m\dvol$ in the weighted Yamabe quotient.  From this standpoint, the following symmetries of the $\mW$-functional are expected.

\begin{prop}
\label{prop:gns_functional_invariant}
Let $(M^n,g,v^m\dvol)$ be a compact smooth metric measure space.  The $\mW$-functional is conformally invariant in its first component:
\begin{equation}
\label{eqn:gns_functional_invariant1}
\mW\left[e^{2\sigma}g,e^{(m+n)\sigma}v^m\dvol\right](w,\tau) = \mW\left[g,v^m\dvol\right]\left(e^{\frac{m+n-2}{2}\sigma}w,\tau\right)
\end{equation}
for all $\sigma,w\in C^\infty(M)$ and $\tau>0$.  It is scale invariant in its second component:
\begin{equation}
\label{eqn:gns_functional_invariant2}
\mW\left[cg,v^m\dvol\right](w,\tau) = \mW\left[g,v^m\dvol\right]\left(c^{\frac{n(m+n-2)}{4(m+n)}}w,c^{-1}\tau\right)
\end{equation}
for all $w\in C^\infty(M)$ and $c,\tau>0$.
\end{prop}

\begin{proof}

\eqref{eqn:gns_functional_invariant1} follows as in Proposition~\ref{prop:yamabe_quotient_invariant}.  \eqref{eqn:gns_functional_invariant2} follows by direct computation.
\end{proof}

We define the energy of a smooth metric measure space by extremizing the $\mW$-functional in the natural way.

\begin{defn}
Let $(M^n,g,v^m\dvol)$ be a compact smooth metric measure space.  Given $\tau>0$, the \emph{$\tau$-energy $\nu[g,v^m\dvol](\tau)\in\bR$} is defined by
\[ \nu[g,v^m\dvol](\tau) = \inf\left\{ \mW(w,\tau) \colon w\in W^{1,2}(M^n,v^m\dvol), \int_M w^{\frac{2(m+n)}{m+n-2}} = 1 \right\} . \]
The \emph{energy $\nu[g,v^m\dvol]\in\bR\cup\{-\infty\}$} is defined by
\[ \nu[g,v^m\dvol] = \inf_{\tau>0} \nu[g,v^m\dvol](\tau) . \]
\end{defn}

Proposition~\ref{prop:gns_functional_invariant} implies that if one defines
\[ \omW[g,v^m\dvol](w,\tau) := \mW[g,v^m\dvol]\left(\tau^{\frac{n(m+n-2)}{4(m+n)}}w,\tau\right), \]
then $\omW\left[cg,v^m\dvol\right](w,\tau)=\mW[g,v^m\dvol](w,c^{-1}\tau)$; i.e.\ $\omW$ has the same symmetries as Perelman's $\mW$-functional.  Indeed, in the case $m=\infty$, the functional $\omW$ agrees with Perelman's $\mW$-functional up to the addition of a constant multiple of the volume $\int w^2\tau^{-n/2}\dvol$.

Proposition~\ref{prop:gns_functional_invariant} also implies that the ($\tau$-)energy is conformally invariant in $w$ (and scale-invariant in $\tau$).

\begin{prop}
\label{prop:energy_invariant}
Let $(M^n,g,v^m\dvol)$ be a compact smooth metric measure space.  Then
\begin{align*}
\nu\left[ce^{2\sigma}g, e^{(m+n)\sigma}v^m\dvol_g\right](c\tau) & = \nu\left[g,v^m\dvol\right](\tau), \\
\nu\left[ce^{2\sigma}g, e^{(m+n)\sigma}v^m\dvol_g\right] & = \nu\left[g,v^m\dvol\right]
\end{align*}
for all $\sigma\in C^\infty(M)$ and for all $c>0$.
\end{prop}

A key fact about the energy is that it is equivalent to the weighted Yamabe constant when the latter is positive.  Indeed, we have the following explicit relationship between the two constants.

\begin{prop}
\label{prop:gns_and_energy}
Let $(M^n,g,v^m\dvol)$ be a compact smooth metric measure space and denote by $\Lambda$ and $\nu$ the weighted Yamabe constant and the energy, respectively.
\begin{enumerate}
\item $\Lambda<0$ if and only if $\nu=-\infty$;
\item $\Lambda=0$ if and only if $\nu=-m$; and
\item $\Lambda>0$ if and only if $\nu>-m$.  Moreover, in this case we have
\begin{equation}
\label{eqn:nu_to_Lambda}
\nu = \frac{2m+n}{2}\left(\frac{2\Lambda}{n}\right)^{\frac{n}{2m+n}} - m ,
\end{equation}
and $w$ is a volume-normalized minimizer of $\Lambda$ if and only if $(w,\tau)$ is a volume-normalized minimizer of $\nu$ for
\[ \tau = \left(\frac{n\int w^{\frac{2(m+n-1)}{m+n-2}}v^{-1}}{2(L_\phi^mw,w)}\right)^{\frac{2(m+n)}{2m+n}} . \]
\end{enumerate}
\end{prop}

\begin{proof}

If $\Lambda<0$, then there is a $w\in C^\infty(M)$ such that $L_\phi^m(w,w)<0$ and $\lV w\rV_{2^\ast}=1$.  It is then clear that $\mW(w,\tau) \to -\infty$ as $\tau\to\infty$.

Suppose $\Lambda\geq0$.  A straightforward calculus exercise shows that if $A,B\geq0$, then
\begin{equation}
\label{eqn:optimize}
\inf_{x>0} \left\{Ax^{2m} + mBx^{-n}\right\} = \frac{2m+n}{2}\left(\frac{2AB^{\frac{2m}{n}}}{n}\right)^{\frac{n}{2m+n}}
\end{equation}
for all $x>0$, with equality if and only if
\begin{equation}
\label{eqn:optimizer}
x = \left(\frac{nB}{2A}\right)^{\frac{1}{2m+n}} .
\end{equation}
It then follows immediately from the definitions of $\Lambda$ and $\nu$ that~\eqref{eqn:nu_to_Lambda} holds.
\end{proof}
\section{Variational formulae for the weighted energy functionals}
\label{sec:variation}

This section is devoted to computing the first variation of the weighted Yamabe quotient and the $\mW$-functional.  We use these formulae in two ways.  First, they identify the Euler--Lagrange equations for these functionals.  Second, the first variation of the $\mW$-functional for variations of the metric and the measure yields a divergence structure for the weighted scalar curvature.  This is the main ingredient in our approach to proving Conjecture~\ref{conj:weighted_obata}.

We first identify the Euler--Lagrange equation for the weighted Yamabe quotient.

\begin{prop}
\label{prop:euler_lagrange}
Let $(M^n,g,v^m\dvol)$ be a compact smooth metric measure space and suppose that $0\leq w\in W^{1,2}(M)$ is a volume-normalized minimizer of the weighted Yamabe constant $\Lambda$.  Then $w$ is a weak solution of
\begin{equation}
\label{eqn:euler_lagrange}
L_\phi^m w + c_1w^{\frac{m+n}{m+n-2}}v^{-1} = c_2 w^{\frac{m+n+2}{m+n-2}},
\end{equation}
where
\begin{align*}
c_1 & = \frac{2m(m+n-1)\Lambda}{n(m+n-2)}\left(\int_M w^{\frac{2(m+n-1)}{m+n-2}}v^{-1}\right)^{-\frac{2m+n}{n}} , \\
c_2 & = \frac{(2m+n-2)(m+n)\Lambda}{n(m+n-2)}\left(\int_M w^{\frac{2(m+n-1)}{m+n-2}}v^{-1}\right)^{-\frac{2m}{n}} .
\end{align*}
\end{prop}

\begin{proof}

Since the weighted conformal Laplacian is self-adjoint, it follows that, as a critical point of the weighted Yamabe functional, $w$ satisfies
\begin{multline*}
L_\phi^mw + \frac{2m(m+n-1)}{n(m+n-2)}\frac{(L_\phi^mw,w)}{\int w^{\frac{2(m+n-1)}{m+n-2}}v^{-1}}w^{\frac{m+n}{m+n-2}}v^{-1} \\
= \frac{(2m+n-2)(m+n)}{n(m+n-2)}\frac{(L_\phi^mw,w)}{\int w^{\frac{2(m+n)}{m+n-2}}} w^{\frac{m+n+2}{m+n-2}} .
\end{multline*}
Using the assumptions on $w$ yields~\eqref{eqn:euler_lagrange}.
\end{proof}

We next identify the Euler--Lagrange equation for the $\mW$-functional.

\begin{lem}
\label{lem:euler_lagrange_w}
Let $(M^n,g,v^m\dvol)$ be a compact smooth metric measure space, fix $\tau>0$, and suppose that $w\in W^{1,2}(M)$ is a nonnegative critical point of the map $\xi\mapsto\mW(\xi,\tau)$ acting on the space of volume-normalized elements of $W^{1,2}(M,v^m\dvol)$.  Then $w$ is a weak solution of
\begin{equation}
\label{eqn:euler_lagrange_w}
\tau^{\frac{m}{m+n}}L_\phi^m w + \frac{m(m+n-1)}{m+n-2}\tau^{-\frac{n}{2(m+n)}}w^{\frac{m+n}{m+n-2}}v^{-1} = c_1 w^{\frac{m+n+2}{m+n-2}}
\end{equation}
for some constant $c_1$.  If additionally $(w,\tau)$ is a minimizer of the energy, then
\begin{equation}
\label{eqn:euler_lagrange_w_c1}
c_1 = \frac{(2m+n-2)(m+n)}{(2m+n)(m+n-2)}\nu[g,v^m\dvol] .
\end{equation}
\end{lem}

\begin{proof}

\eqref{eqn:euler_lagrange_w} follows immediately from the definition of $\mW$.  If $(w,\tau)$ is a critical point of the map $(w,\tau)\mapsto\mW(w,\tau)$, then
\[ \tau^{\frac{2m+n}{2(m+n)}}\left(L_\phi^mw,w\right) = \frac{n}{2}\int_M w^{\frac{2(m+n-1)}{m+n-2}}v^{-1} . \]
Using this identity and integrating~\eqref{eqn:euler_lagrange_w} against $wv^m\dvol$ yields the result.
\end{proof}

In order to work towards the weighted Obata theorem, we require the first variation of the $\mW$-functional through changes in the metric and the measure.  In order to perform this computation, it is convenient to introduce the total weighted scalar curvature functional.

\begin{defn}
Let $M^n$ be a compact manifold, fix $m\in[0,\infty]$ and $\mu\in\bR$, and denote by $\Met(M)$ and $\kM_+$ the spaces of smooth Riemannian metrics and smooth measures, respectively, on $M$.  The \emph{total weighted scalar curvature functional}
\[ \mW[\mu]\colon\Met(M)\times\kM_+\to\bR \]
is the functional
\begin{equation}
\label{eqn:total_weighted_scalar_curvature_functional}
\mW[\mu]\left(g,v^m\dvol\right) := \int_M \left( R_\phi^m + m\mu (v^{-1}-1)\right) ,
\end{equation}
where $R_\phi^m$ is the weighted scalar curvature of $(M^n,g,v^m\dvol)$ and integration is performed with respect to $v^m\dvol$.
\end{defn}

It is straightforward to check that when $\mu>0$, the total weighted scalar curvature functional is effectively the $\mW$-functional.  Indeed,
\[ \mW[g,v^m\dvol](1,\tau) = \tau^{\frac{m}{m+n}}\int_M \left[ \frac{m+n-2}{4(m+n-1)}R_\phi^m + m\tau^{-\frac{2m+n}{2(m+n)}}v^{-1} - m \right] , \]
so that the claimed equivalence follows readily from Proposition~\ref{prop:gns_functional_invariant}.

Computations carried out in~\cite{Case2010a} yield the first variation of the total weighted scalar curvature functional .  It is convenient to write this in two equivalent ways.

\begin{prop}
\label{prop:var_total_weighted_scalar_curvature}
Let $M^n$ be a compact Riemannian manifold, fix $m\in[0,\infty]$ and $\mu\in\bR$, and let $g(s)$ and $\phi(s)$ be a smooth one parameter family of metrics and functions, respectively, with $\int e^{-\phi(s)}\dvol_{g(s)}$ constant.  Set $g=g(0), h=g^\prime(0), \phi=\phi(0),$ and $\psi=\phi^\prime(0)$.  Then the first variation
\[ \delta\mW := \left.\frac{d}{ds}\right|_{s=0}\mW[\mu]\left(g(s),e^{-\phi(s)}\dvol_{g(s)}, m\right) \]
of the total weighted scalar curvature functional is
\begin{equation}
\label{eqn:var_schouten}
\begin{split}
\delta\mW & = -\int_M \left\lp \Ric_\phi^m - \frac{R_\phi^m}{2(m+n-1)}g, h + \frac{2}{m}\psi g\right\rp \\
& \quad - \frac{m+n-2}{m+n-1}\int_M \left(R_\phi^m + \frac{m(m+n-1)\mu}{m+n-2} v^{-1}\right)\left(\frac{m-1}{m}\psi - \frac{1}{2}\tr h\right) ,
\end{split}
\end{equation}
or equivalently,
\begin{equation}
\label{eqn:var_trfree}
\begin{split}
\delta\mW & = -\int_M \left\lp \Ric_\phi^m - \frac{R_\phi^m}{m+n}g - \frac{m\mu v^{-1}}{2(m+n)}g, h + \frac{2}{m}\psi g\right\rp \\
& \quad - \frac{m+n-2}{m+n}\int_M \left(R_\phi^m + \frac{m(m+n-1)\mu}{m+n-2}v^{-1}\right)\left(\psi-\frac{1}{2}\tr h\right) ,
\end{split}
\end{equation}
where all geometric invariants are computed using $(M^n,g,v^m\dvol)$.
\end{prop}

\begin{proof}

In~\cite[Proposition~4.19]{Case2010a} it is shown that
\[ \left.\frac{d}{ds}\right|_{s=0}\int R_\phi^m = - \int_M \left[ \left\lp\Ric_\phi^m-\frac{1}{2}R_\phi^mg,h\right\rp + \left(R_\phi^m-\frac{2}{m}\Delta_\phi\phi\right)\psi\right] . \]
The result then follows from the trivial computation of the first variation of $\int v^{-1}$ and straightforward algebraic manipulations.
\end{proof}

There are a few comments to make about Proposition~\ref{prop:var_total_weighted_scalar_curvature}.  First, using the correspondence between the total weighted scalar curvature functional and the $\mW$-functional, it recovers Lemma~\ref{lem:euler_lagrange_w} when restricted to conformal variations of smooth metric measure spaces.  This is because the metric $v^{-2}g$ is always fixed for such a variation, and hence $h+\frac{2}{m}\psi g=0$.  Second, $(M^n,g,v^m\dvol)$ is a critical point of the total weighted scalar curvature functional if and only if there are constants $\lambda$ and $\mu$ such that
\begin{align}
\label{eqn:qe1_lambda} \Ric_\phi^m - \frac{R_\phi^m}{2(m+n-1)}g & = \frac{(m+n-2)\lambda}{2(m+n-1)}g \\
\label{eqn:qe1_mu} \Ric_\phi^m - \frac{R_\phi^m}{m+n}g & = \frac{m\mu v^{-1}}{2(m+n)}g .
\end{align}
Third, Proposition~\ref{prop:var_total_weighted_scalar_curvature} yields the following Bianchi-type formula.  Indeed, this result establishes that~\eqref{eqn:qe1_lambda} implies~\eqref{eqn:qe1_mu} and vice versa; see Section~\ref{sec:uniqueness} for details.

\begin{prop}
\label{prop:weighted_bianchi}
Let $(M^n,g,v^m\dvol)$ be a smooth metric measure space, fix $\mu\in\bR$, and denote
\begin{subequations}
\label{eqn:weighted_trfree_ric}
\begin{align}
\label{eqn:weighted_trfree_ric_notilde} E_\phi^m & := \Ric_\phi^m - \frac{R_\phi^m}{m+n}g , \\
\label{eqn:weighted_trfree_ric_tilde} \widetilde{E_\phi^m} & := E_\phi^m - \frac{m\mu v^{-1}}{2(m+n)}g .
\end{align}
\end{subequations}
It holds that
\begin{equation}
\label{eqn:weighted_bianchi}
\delta_\phi\left(\widetilde{E_\phi^m}\right) = \frac{1}{m}\tr\left(\widetilde{E_\phi^m}\right)d\phi + \frac{m+n-2}{2(m+n)}d\left(R_\phi^m + \frac{m(m+n-1)\mu v^{-1}}{m+n-2}\right) .
\end{equation}
\end{prop}

\begin{proof}

Let $X$ be a compactly-supported vector field on $M$ and let $\{f_s\colon M\to M\}$, $s\in(-\varepsilon,\varepsilon)$ be the one-parameter family of diffeomorphisms such that $f_0=\Id$ and $\frac{\partial}{\partial s}\rv_{s=0}f_s=X$.  Set $g(s)=f_s^\ast g$ and $v(s)=f_s^\ast v$.  Since the total weighted scalar curvature functional is a geometric invariant, $\delta\mW=0$.  Then~\eqref{eqn:var_trfree} yields
\[ \int_M\left\lp \widetilde{E_\phi^m}, L_Xg + \frac{2}{m}X\phi\,g\right\rp = \frac{m+n-2}{m+n}\int_M \left(R_\phi^m + \frac{m(m+n-1)\mu v^{-1}}{m+n-2}\right)\,\delta_\phi X . \]
The conclusion follows by integration by parts.
\end{proof}
\section{Euclidean space as the model space}
\label{sec:model}

As described in the introduction, Del Pino and Dolbeault~\cite{DelPinoDolbeault2002} have already solved the weighted Yamabe problem on Euclidean space by giving a complete classification of the positive critical points of the weighted Yamabe quotient on $\left(\bR^n,dx^2,1^m\dvol\right)$.  In this section, we discuss their result in terms of the $\mW$-functional.  In particular, this allows us to clearly identify the relationship between the parameter $\tau$ and concentration for minimizers, and motivates Proposition~\ref{prop:sup_estimate}, which gives an important estimate needed in Section~\ref{sec:existence}.

Fix $n\geq 3$ and $m\in[0,\infty]$.  Given any $x_0\in\bR^n$ and $\tau>0$, define the function $w_{x_0,\tau}\in C^\infty(\bR^n)$ by
\begin{equation}
\label{eqn:modelw}
w_{x_0,\tau}(x) = \tau^{-\frac{n(m+n-2)}{4(m+n)}}\left(1+\frac{(m+n-1)\lv x-x_0\rv^2}{(m+n-2)^2\tau}\right)^{-\frac{m+n-2}{2}} .
\end{equation}
A straightforward computation shows that
\begin{equation}
\label{eqn:modelw_volume_preserved}
\int_{\bR^n} w_{x_0,\tau}^{\frac{2(m+n)}{m+n-2}} 1^m\dvol = \int_{\bR^n} w_{0,1}^{\frac{2(m+n)}{m+n-2}}1^m\dvol =: V
\end{equation}
and
\begin{equation}
\label{eqn:modelw_laplacian}
-\tau^{\frac{m}{m+n}}\Delta w_{x_0,\tau} + \frac{m(m+n-1)}{m+n-2}\tau^{-\frac{n}{2(m+n)}}w_{x_0,\tau}^{\frac{m+n}{m+n-2}} = \frac{(m+n)(m+n-1)}{m+n-2} w^{\frac{m+n+2}{m+n-2}} .
\end{equation}
Lemma~\ref{lem:euler_lagrange_w} together with~\eqref{eqn:modelw_volume_preserved} and~\eqref{eqn:modelw_laplacian} imply that $w_{x_0,\tau}$ is a critical point of the map $w\mapsto\mW(w,\tau)$ subject to the constraint $\int w^{\frac{2(m+n)}{m+n-2}}=V$.  Indeed, Del Pino and Dolbeault showed that the \emph{only} nonnegative critical points in $W^{1,2}(\bR^n)$ subject to this constraint are the functions $w_{x_0,\tau}$; see~\cite[Theorem~4]{DelPinoDolbeault2002}.


For our purposes, the most interesting aspect of the solutions~\eqref{eqn:modelw_volume_preserved} is the way in which they concentrate.  It is straightforward to check that
\begin{equation}
\label{eqn:modelw_sup}
\sup_{x\in\bR^n} w_{x_0,\tau}(x) = w_{x_0,\tau}(x_0) = \tau^{-\frac{n(m+n-2)}{4(m+n)}}
\end{equation}
and also that, for any $y\not=x_0$,
\begin{equation}
\label{eqn:modelw_blowup}
\lim_{\tau\to0^+} w_{x_0,\tau}(y) = 0 .
\end{equation}
These observations tell us two things.  First, concentration occurs for the solutions to the weighted Yamabe problem on Euclidean space as $\tau\to0^+$.  Second, when $m>0$, the critical points of the volume-constrained functional $w\mapsto\mW(w,\tau)$ with $\tau$ fixed do not concentrate.  These phenomena persist in the curved case.

\begin{prop}
\label{prop:sup_estimate}
Let $(M^n,g,v^m\dvol)$ be a compact smooth metric measure space with $m>0$, fix $\tau_0$, and suppose that $w\in C^\infty(M)$ is a positive function such that
\[ \mW(w,\tau) = \nu[g,v^m\dvol](\tau) \quad\text{and}\quad \int_M w^{\frac{2(m+n)}{m+n-2}} = 1 \]
for some $0<\tau\leq\tau_0$.  Then there exist constants $C_1,C_2>0$ depending only on $(M^n,g,v^m\dvol)$ and $\tau_0$ such that
\begin{equation}
\label{eqn:sup_estimate}
C_1 \leq \sup_{x\in M} \tau^{\frac{n(m+n-2)}{4(m+n)}}w(x) \leq C_2 .
\end{equation}
\end{prop}

\begin{proof}

We denote throughout this proof by $C,C_1,C_2>0$ constants depending only on $(M^n,g,v^m\dvol)$ and $\tau_0$ whose values may change from line to line.

To begin, H\"older's inequality and the normalization of $w$ imply that
\begin{equation}
\label{eqn:energy_estimate}
\tau^{\frac{m}{m+n}}\left(L_\phi^mw,w\right) \geq -C .
\end{equation}
By Lemma~\ref{lem:euler_lagrange_w} there exists a constant $c\in\bR$ such that
\begin{equation}
\label{eqn:euler_lagrange_sup_estimate}
\tau^{\frac{m}{m+n}}L_\phi^m w + \frac{m(m+n-1)}{m+n-2}\tau^{-\frac{n}{2(m+n)}}w^{\frac{m+n}{m+n-2}}v^{-1} = cw^{\frac{m+n+2}{m+n-2}} ,
\end{equation}
while the fact that $w$ realizes $\nu[g,v^m\dvol](\tau)$ implies that
\begin{equation}
\label{eqn:sup_estimate_minimizer}
\tau^{\frac{m}{m+n}}\left(L_\phi^mw,w\right) + m\tau^{-\frac{n}{2(m+n)}}\int_M w^{\frac{2(m+n-1)}{m+n-2}}v^{-1} = \nu(\tau) + m .
\end{equation}
As a consequence, we have that
\begin{align*}
c-\nu-m & = \frac{m}{m+n-2}\tau^{-\frac{m}{2(m+n)}}\int_M w^{\frac{2(m+n-1)}{m+n-2}}v^{-1} \geq 0 , \\
\nu + m - \frac{m+n-2}{m+n-1}c & = \frac{1}{m+n-1}\tau^{\frac{m}{m+n}}\left(L_\phi^mw,w\right) .
\end{align*}
In particular, it follows from~\eqref{eqn:energy_estimate} that $C_1\leq e^c\leq C_2$.  Furthermore, \eqref{eqn:energy_estimate}, \eqref{eqn:sup_estimate_minimizer}, and the assumption $m>0$ together imply that the lower bound in~\eqref{eqn:sup_estimate} holds.

The upper bound for $w$ is a consequence of the above estimates and the fact that the nonlinearity in~\eqref{eqn:euler_lagrange_sup_estimate} is subcritical when $m>0$.  More precisely, Proposition~\ref{prop:gns_functional_invariant} implies that the rescaling
\[ \left(\tilde g,\tilde w\right) := \left( \tau^{-1}g, \tau^{\frac{n(m+n-2)}{4(m+n)}}w \right) \]
is such that $\tilde w$ satisfies
\[ \widetilde{L_\phi^m}\tilde w + \frac{m(m+n-1)}{m+n-2}\tilde w^{\frac{m+n}{m+n-2}}v^{-1} = c\tilde w^{\frac{m+n+2}{m+n-2}} \quad\text{and}\quad \int_M \tilde w^{\frac{2(m+n)}{m+n-2}} = 1 \]
with respect to $(M^n,\tilde g,v^m\dvol_{\tilde g})$.  Since $\tau\leq\tau_0$, it follows that
\[ \sup_{x\in M} \left|\widetilde{R_\phi^m}(x)\right| \leq \tau_0\sup_{x\in M} \left|R_\phi^m\right| , \]
and hence $\tilde w$ satisfies
\begin{equation}
\label{eqn:tilde_euler_lagrange_ineq}
-\tilde\Delta_\phi\tilde w \leq c\tilde w^{\frac{m+n+2}{m+n-2}} - C\tilde w .
\end{equation}
On the other hand, standard Sobolev inequalities (cf.\ \cite{Hebey1999}) imply that
\begin{equation}
\label{eqn:tilde_sobolev}
\left(\int_M f^{\frac{2n}{n-2}}\dvol_{\tilde g}\right)^{\frac{n-2}{n}} \leq C_1\int_M\lv\tilde\nabla f\rv_{\tilde g}^2\dvol_{\tilde g} + C_2\int_M f^2\dvol_{\tilde g}
\end{equation}
for all $f\in C^\infty(M)$.  Since $m>0$, the nonlinearity of~\eqref{eqn:tilde_euler_lagrange_ineq} is subcritical, and in particular it follows readily from~\eqref{eqn:tilde_sobolev} and Moser iteration that $\sup\tilde w\leq C$, which is equivalent to the upper bound in~\eqref{eqn:sup_estimate}.
\end{proof}
\section{The existence of minimizers}
\label{sec:existence}

We are able to establish the existence of minimizers of the weighted Yamabe quotient.  Like the proof of the corresponding result for the Yamabe Problem (cf.\ \cite{Aubin1976,Trudinger1968}), it is convenient to separate the proof of Theorem~\ref{thm:blow_up} into two cases.  First, in Proposition~\ref{prop:existence_negative}, we show by a direct compactness argument that if the weighted Yamabe constant is negative, then there exists a smooth, positive minimizer.  Second, in Proposition~\ref{prop:existence_positive}, we show that as $\tau\to0$, the $\tau$-energy of a smooth metric measure space with nonnegative weighted Yamabe constant tends to the weighted Yamabe constant of Euclidean space.  The proof of Proposition~\ref{prop:existence_positive} presented here uses a direct blow-up argument.  These two results together yield Theorem~\ref{thm:blow_up}.

\begin{prop}
\label{prop:existence_negative}
Let $(M^n,g,v^m\dvol)$ be a compact smooth metric measure space with $m>0$ and negative weighted Yamabe constant.  Then there exists a positive function $w\in C^\infty(M)$ such that
\[ \mQ(w) = \Lambda[g,v^m\dvol] . \]
\end{prop}

\begin{proof}

Let $\{w_k\}\subset C^\infty(M)$ be a volume-normalized minimizing sequence of the weighted Yamabe constant.  Since the weighted Yamabe constant is negative,
\[ 0 > \left(L_\phi^mw_k,w_k\right) \geq \lV\nabla w_k\rV_2^2 - C\lV w_k\rV_2^2 \]
for $k$ sufficiently large and $C$ a constant depending only on $(M^n,g,v^m\dvol)$.  Since the functions $w_k$ are volume-normalized, H\"older's inequality and the above display imply that $\{w_k\}$ is uniformly bounded in $W^{1,2}(M)$.  Hence there is a uniform constant $C>0$ such that
\begin{equation}
\label{eqn:necessary_lower_bound}
\int_M w_k^{\frac{2(m+n-1)}{m+n-2}}v^{-1} \geq C .
\end{equation}
Since $m>0$, the embedding $W^{1,2}(M)\subset L^{\frac{2(m+n)}{m+n-2}}$ is compact, and hence there is a $w\in W^{1,2}(M)$ such that $w_k$ converges weakly in $W^{1,2}$ to $w$ and strongly in $L^{\frac{2(m+n)}{m+n-2}}$ to $w$.  By construction, $w$ minimizes the weighted Yamabe constant, and also satisfies the lower bound~\eqref{eqn:necessary_lower_bound}.  It then follows from Proposition~\ref{prop:euler_lagrange} that $w$ is a weak solution to
\[ L_\phi^m w + c_1v^{-1}w^{\frac{m+n}{m+n-2}} = c_2w^{\frac{m+n+2}{m+n-2}} \]
for explicit constants $c_1,c_2$.  Since $1\leq\frac{m+n}{m+n-2}\leq\frac{m+n+2}{m+n-2}<\frac{n+2}{n-2}$ when $m>0$, the usual elliptic regularity argument for subcritical equations allows us to conclude that $w$ is in fact smooth and positive, as desired.
\end{proof}

To study the case when the weighted Yamabe constant is nonnegative requires us to study minimizers of the $\mW$-functional.

\begin{lem}
\label{lem:minimize_tau_energy}
Let $(M^n,g,v^m\dvol)$ be a compact smooth metric measure space with $m>0$ and fix $\tau>0$.  Then there exists a positive function $w\in C^\infty(M)$ such that
\[ \mW[g,v^m\dvol](w,\tau) = \nu(\tau) \quad\text{and}\quad \int_M w^{\frac{2(m+n)}{m+n-2}} = 1 . \]
\end{lem}

\begin{proof}

By H\"older's inequality, there is a constant $C$ such that
\[ \mW[g,v^m\dvol](w,\tau) \geq \tau^{\frac{m}{m+n}}\lV\nabla w\rV_2^2 - C \]
for all volume-normalized $w\in W^{1,2}(M,v^m\dvol)$.  In particular, $\nu(\tau)>-\infty$ and any volume-normalized minimizing sequence $\{w_k\}$ of $\nu(\tau)$ is uniformly bounded in $W^{1,2}(M,v^m\dvol)$.  Since $m>0$, the embedding $W^{1,2}\subset L^{\frac{2(m+n)}{m+n-2}}$ is compact.  Hence, taking a subsequence if necessary, we see that $w_k$ converges to a volume-normalized $w\in W^{1,2}(M)$ such that
\[ \mW[g,v^m\dvol](w,\tau) = \nu(\tau) . \]
By Lemma~\ref{lem:euler_lagrange_w}, $w$ is a weak solution to the subcritical elliptic PDE~\eqref{eqn:euler_lagrange_w}, and hence $w$ is smooth and positive.
\end{proof}

\begin{prop}
\label{prop:existence_positive}
Let $(M^n,g,v^m\dvol)$ be a compact smooth metric measure space with $m>0$.  Then
\[ \lim_{\tau\to 0} \nu[g,v^m\dvol](\tau) = \nu[\bR^n,dx^2,1^m\dvol] . \]
\end{prop}

\begin{proof}

We first show that $\liminf_{\tau\to0}\nu(\tau)\geq\Lambda[\bR^n,dx^2,1^m\dvol]$.  To that end, let $\{\tau_i\}$ be a decreasing sequence of positive numbers tending to zero such that $\nu(\tau_i)$ converges.  By Lemma~\ref{lem:minimize_tau_energy}, there are positive functions $w_i\in C^\infty(M)$ such that $\mW(w_i,\tau_i)=\nu(\tau_i)$ and $\int w_i^{\frac{2(m+n)}{m+n-2}}=1$.  It follows from Lemma~\ref{lem:euler_lagrange_w} that there are constants $c_i$ such that
\begin{equation}
\label{eqn:existence_positive_el}
\tau_i^{\frac{m}{m+n}}L_\phi^mw_i + \frac{m(m+n-1)}{m+n-2}\tau_i^{-\frac{n}{2(m+n)}}w_i^{\frac{m+n}{m+n-2}}v^{-1} = c_i w_i^{\frac{m+n+2}{m+n-2}} .
\end{equation}
The proof of Proposition~\ref{prop:sup_estimate} shows that the constants $c_i$ are uniformly bounded above and below.  Using the choice of $w_i$ and~\eqref{eqn:existence_positive_el} together with the formula for the derivative of the map $\tau\mapsto\mW(w,\tau)$ yields
\[ m + \nu(\tau_i) - \frac{(2m+n)(m+n-2)}{(2m+n-2)(m+n)}c_i = \frac{2\tau_i}{2m+n-2}\frac{d}{dt}\mW(w_i,t)\big\rv_{t=\tau_i} . \]
In particular, we have that
\begin{equation}
\label{eqn:nu_to_c}
\lim_{i\to\infty} \nu(\tau_i) + m - \frac{(2m+n)(m+n-2)}{(2m+n-2)(m+n)}c_i = 0 .
\end{equation}
Next, set $\tilde g_i=\tau_i^{-1}g$ and $\tilde w_i=\tau_i^{\frac{n(m+n-2)}{4(m+n)}}w_i$, in terms of which~\eqref{eqn:existence_positive_el} becomes
\begin{equation}
\label{eqn:existence_positive_el_rescale}
\widetilde{L_\phi^m}\tilde w_i + \frac{m(m+n-1)}{m+n-2}\tilde w_i^{\frac{m+n}{m+n-2}}v^{-1} = c_i\tilde w_i^{\frac{m+n+2}{m+n-2}} ;
\end{equation}
observe also that the normalization of $w_i$ persists,
\[ \int_M \tilde w_i^{\frac{2(m+n)}{m+n-2}} v^m\dvol_{\tilde g_i} = 1 . \]
By Proposition~\ref{prop:sup_estimate} there are constants $C_1,C_2>0$ independent of $i$ such that $C_1\leq\sup\tilde w_i\leq C_2$.  Let $x_i\in M$ be such that $\tilde w_i(x_i)=\sup\tilde w_i$.  By taking a subsequence if necessary, we may suppose that the points $x_i$ to $x_0$ and that $\sup\tilde w_i$ and $c_i$ converge as $i\to\infty$.  Since $v(x_0)$ is positive and finite, we may assume without loss of generality that $v(x_0)=1$; this is a straightforward consequence of~\eqref{eqn:gns_functional_invariant1}.

Now, given any fixed normal coordinate chart $U$ of $x_0$, it follows that $C_1\leq\sup_U\tilde w_i\leq C_2$.  Indeed, it follows from~\eqref{eqn:existence_positive_el_rescale} and elliptic regularity theory that the $C^{2,\alpha}$-norms of $\tilde w_i\rv_U$ are uniformly bounded.  We may thus extract a subsequence such that $\tilde w_i\rv_U$ converges in $C^{2,\alpha}$ to a nonnegative function $\tilde w\in C^\infty(\bR^n)$ which satisfies $\tilde w(0)>0$ and $\int \tilde w^{\frac{2(m+n)}{m+n-2}}\leq1$ and
\[ -\Delta\tilde w + \frac{m(m+n-1)}{m+n-2}\tilde w^{\frac{m+n}{m+n-2}} = c\tilde w^{\frac{m+n+2}{m+n-2}} \]
for $c=\lim_{i\to\infty} c_i$.  It follows from Theorem~\ref{thm:dd} that
\[ \frac{(2m+n)(m+n-2)}{(2m+n-2)(m+n)}c \geq m + \nu[\bR^n,dx^2,\dvol], \]
and hence~\eqref{eqn:nu_to_c} implies that $\liminf_{\tau\to0}\nu(\tau)\geq\Lambda[\bR^n,dx^2,1^m\dvol]$.

Let us now show that $\lim\sup_{\tau\to0}\nu(\tau)\leq\Lambda[\bR^n,dx^2,1^m\dvol]$.  Fix a point $p\in M$ and let $\{x_i\}$ be normal coordinates in some fixed neighborhood $U$ of $p=(0,\dotsc,0)$.  Let $\varepsilon>0$ be such that $B(p,2\varepsilon)\subset U$.  Let $\eta\colon M\to[0,1]$ be a cutoff function such that $\eta\equiv 1$ on $B(p,\varepsilon)$ and $\supp\eta\subset B(p,2\varepsilon)$.  For each $0<\tau<1$, define $f_\tau\colon M\to\bR$ by $f_\tau(x_1,\dotsc,x_n) = \eta w_{0,\tau}(x_1,\dotsc,x_n)$, and set $\tilde f_\tau = V_\tau^{-\frac{m+n-2}{2(m+n)}}f_\tau$ for
\[ V_\tau = \int_M f_\tau^{\frac{2(m+n)}{m+n-2}} . \]
By the choice of $f_\tau$, the constants $V_\tau$ are uniformly bounded away from zero.  With $V$ as in~\eqref{eqn:modelw_volume_preserved}, we have that
\begin{multline*}
V^{\frac{m+n-2}{m+n}}\left(\nu[\bR^n,dx^2,1^m\dvol]+m\right) \\ = \tau^{\frac{m}{m+n}}\int_{\bR^n}\lv\nabla w_{0,\tau}\rv^2 + mV^{-\frac{1}{m+n}}\tau^{-\frac{n}{2(m+n)}}\int_{\bR^n}w_{0,\tau}^{\frac{2(m+n-1)}{m+n-2}} .
\end{multline*}
Computing as in~\cite[Lemma~3.4]{LeeParker1987}, it is easy to see that
\[ \mW[g,v^m\dvol](\tilde f_\tau,\tau) \leq \nu[\bR^n,dx^2,1^m\dvol]\left(1+C_1\varepsilon\right)\left(1+C_2\tau^{\frac{n-2}{2}}\right) , \]
where the constant $C_1>0$ depends only on $(M^n,g,v^m\dvol)$ and the constant $C_2>0$ depends only on $(M^n,g,v^m\dvol)$ and $\varepsilon$.  Taking $\tau\to0$ and then $\varepsilon\to0$ yields the desired result.
\end{proof}

We can now prove Theorem~\ref{thm:blow_up}.

\begin{proof}[Proof of Theorem~\ref{thm:blow_up}]

In the case $m=0$, Theorem~\ref{thm:blow_up} is already contained in Aubin's work~\cite{Aubin1976} on the Yamabe Problem.  If $m>0$, then Proposition~\ref{prop:gns_and_energy} and Proposition~\ref{prop:existence_positive} that~\eqref{eqn:weighted_yamabe_estimate} holds.

Suppose now that strict inequality holds in~\eqref{eqn:weighted_yamabe_estimate}.  If the weighted Yamabe constant is negative, then the result follows from Proposition~\ref{prop:existence_negative}.  If the weighted Yamabe constant is zero, then Proposition~\ref{prop:signs} and its proof yields the result.  If the weighted Yamabe constant is positive, then Proposition~\ref{prop:signs} and the proof of Proposition~\ref{prop:gns_and_energy} imply that $\nu(\tau)\to\infty$ as $\tau\to\infty$.  Hence Proposition~\ref{prop:existence_positive} implies that there is a $\tau\in(0,\infty)$ such that $\nu(\tau)=\nu$.  The result is then an immediate consequence of Proposition~\ref{prop:gns_and_energy} and Lemma~\ref{lem:minimize_tau_energy}.
\end{proof}
\section{Towards characterizing equality in~\eqref{eqn:weighted_yamabe_estimate}}
\label{sec:yamabe}

We now turn our attention to Theorem~\ref{thm:integral_characterization}.  The key observation is the following relation between the weighted Yamabe constants of $(M^n,g,v^m\dvol)$ and $(M^n,g,v^{m+1}\dvol)$.

\begin{thm}
\label{thm:constants_reln}
Let $(M^n,g,v^m\dvol)$ be a compact smooth metric measure space with nonnegative weighted Yamabe constant, and suppose that there exists a smooth, positive minimizer $w$ of the weighted Yamabe constant.  Then
\begin{equation}
\label{eqn:constants_reln}
\Lambda[g,v^{m+1}\dvol] \leq \frac{\Lambda[\bR^n,dx^2,1^{m+1}\dvol]}{\Lambda[\bR^n,dx^2,1^m\dvol]} \Lambda[g,v^m\dvol] .
\end{equation}

Moreover, if equality holds in~\eqref{eqn:constants_reln}, then $\hat v=w^{\frac{2}{m+n-2}}v$ satisfies
\begin{align}
\label{eqn:constants_rigidity_lapl} -\hat\Delta\hat v^{m+1} & = C\Lambda[g,v^m\dvol]\left(\hat v^{m+1} - \hat v^m\int_M \hat v^{m+1}\dvol_{\hat g}\right) , \\
\label{eqn:constants_rigidity_norms} \frac{m(2m+n)(m+n-1)}{(2m+n-2)(m+n)} & = m\left(\int_M\hat v^{m-1}\dvol_{\hat g}\right)\left(\int_M\hat v^{m+1}\dvol_{\hat g}\right) ,
\end{align}
for $\hat g=w^{\frac{4}{m+n-2}}g$ and
\[ C = \frac{4(m+1)(m+n)(2m+n-2)}{n(m+n-2)^2}\left(\int_M\hat v^{m-1}\dvol_{\hat g}\right)^{-\frac{2m}{n}} . \]
\end{thm}

The basic idea of the proof is that, by weight considerations (cf.\ Theorem~\ref{thm:dd}), if $w$ is a minimizer of the weighted Yamabe constant of $(M^n,g,v^m\dvol)$, then $w^{\frac{m+n-1}{m+n-2}}$ is a natural test function for estimating the weighted Yamabe constant of $(M^n,g,v^{m+1}\dvol)$.  The key step in realizing this idea is the following computation.

\begin{prop}
\label{prop:increment_m_by_1}
Let $(M^n,g)$ be a compact Riemannian manifold and fix $m\in[0,\infty]$ and a positive function $v\in C^\infty(M)$.  Given any $w\in C^\infty(M)$, it holds that
\begin{equation}
\label{eqn:change_m}
\left( L_\phi^{m+1} w^{\frac{m+n-1}{m+n-2}}, w^{\frac{m+n-1}{m+n-2}}\right) = \frac{(m+n-1)^2}{(m+n)(m+n-2)}\left( L_\phi^mw, w^{\frac{m+n}{m+n-2}}v\right) ,
\end{equation}
where the left and the right hand side are defined relative to the smooth metric measure spaces $(M^n,g,v^{m+1}\dvol)$ and $(M^n,g,v^{m}\dvol)$, respectively.
\end{prop}

\begin{proof}

First note that both sides of~\eqref{eqn:change_m} are conformally invariant in the sense of smooth metric measure spaces, so that, as in the proof of Proposition~\ref{prop:monotone_in_m}, we may assume $v=1$.  The result then follows immediately from the identity
\[ \left|\nabla w^{\frac{m+n-1}{m+n-2}}\right|^2 = \frac{(m+n-1)^2}{(m+n)(m+n-2)}\left\lp\nabla w,\nabla w^{\frac{m+n}{m+n-2}}\right\rp \]
and the definition of the weighted conformal Laplacian.
\end{proof}

\begin{proof}[Proof of Theorem~\ref{thm:constants_reln}]

By Proposition~\ref{prop:increment_m_by_1}, we have that for all positive $w\in C^\infty(M)$ which are volume-normalized with respect to $(M^n,g,v^m\dvol)$,
\begin{equation}
\label{eqn:constants_reln_form1}
\mQ[g,v^{m+1}\dvol]\left(w^{\frac{m+n-1}{m+n-2}}\right) = \frac{(m+n-1)^2}{(m+n)(m+n-2)}\frac{\left(L_\phi^mw,w^{\frac{m+n}{m+n-2}}v\right)}{\left(\int w^{\frac{2(m+n+1)}{m+n-2}}v\right)^{\frac{2m+n}{n}}},
\end{equation}
where all quantities on the right hand side are defined in terms of $(M^n,g,v^m\dvol)$.  Proposition~\ref{prop:euler_lagrange} implies that if $w$ is a minimizer of $\Lambda:=\Lambda[g,v^m\dvol]$, then
\begin{equation}
\label{eqn:constants_reln_form2}
\begin{split}
\left(L_\phi^mw,w^{\frac{m+n}{m+n-2}}v\right) & = -\frac{2m(m+n-1)\Lambda}{n(m+n-2)}\left(\int_M w^{\frac{2(m+n-1)}{m+n-2}}v^{-1}\right)^{-\frac{2m+n}{n}} \\
& \quad + \frac{(2m+n-2)(m+n)\Lambda}{n(m+n-2)}\frac{\int w^{\frac{2(m+n+1)}{m+n-2}}v}{\left(\int w^{\frac{2(m+n-1)}{m+n-2}}v^{-1}\right)^{\frac{2m}{n}}} .
\end{split}
\end{equation}
Combining~\eqref{eqn:constants_reln_form1} and~\eqref{eqn:constants_reln_form2} yields
\begin{equation}
\label{eqn:constants_reln_form3}
\mQ[g,v^{m+1}\dvol]\left(w^{\frac{m+n-1}{m+n-2}}\right) = \frac{(m+n-1)^2\Lambda}{n(m+n)(m+n-2)^2}\Phi(x)
\end{equation}
for
\begin{align*}
\Phi(x) & = (2m+n-2)(m+n)x^{-\frac{2m}{n}} - 2m(m+n-1)x^{-\frac{2m+n}{n}}, \\
x & = \left(\int_M w^{\frac{2(m+n-1)}{m+n-2}}v^{-1}\right)\left(\int_M w^{\frac{2(m+n+1)}{m+n-2}}v\right) .
\end{align*}
H\"older's inequality and the volume-normalization of $w$ imply that $x\geq 1$.  On the other hand, a straightforward calculus exercise reveals that
\begin{equation}
\label{eqn:Phi_estimate}
\Phi(x) \leq \frac{(2m+n-2)(m+n)n}{2m+n}\left(\frac{(2m+n-2)(m+n)}{(2m+n)(m+n-1)}\right)^{\frac{2m}{n}}
\end{equation}
with equality if and only if
\begin{equation}
\label{eqn:x_condn}
mx = \frac{m(2m+n)(m+n-1)}{(2m+n-2)(m+n)} .
\end{equation}
It follows from~\eqref{eqn:dd_yamabe} and~\eqref{eqn:constants_reln_form3} that
\[ \Lambda[g,v^{m+1}\dvol] \leq \frac{\Lambda[\bR^n,dx^2,1^{m+1}\dvol]}{\Lambda[\bR^n,dx^2,1^m\dvol]}\Lambda[g,v^m\dvol] , \]
establishing~\eqref{eqn:constants_reln}.

Suppose now that equality holds in~\eqref{eqn:constants_reln}.  Then $w^{\frac{m+n-1}{m+n-2}}$ is a minimizer of $\Lambda[g,v^{m+1}\dvol]$.  In particular, equality holds in~\eqref{eqn:Phi_estimate}, and hence~\eqref{eqn:x_condn} holds.  Furthermore, conformal invariance implies that the constant function $1$ is a volume-normalized minimizer for the weighted Yamabe constants of both $(M^n,\hat g,\hat v^{m}\dvol)$ and $(M^n,\hat g,\hat v^{m+1}\dvol)$.  Since the integrals appearing in the definition of $x$ are conformally invariant, it follows that~\eqref{eqn:constants_rigidity_norms} holds.  On the other hand, Proposition~\ref{prop:euler_lagrange} applied to the minimizer $1$ yields
\begin{align*}
\frac{m+n-2}{4(m+n-1)}\widehat{R_\phi^m} & = \frac{(2m+n-2)(m+n)\Lambda_m}{n(m+n-2)}\left(\int_M \hat v^{m-1}\dvol_{\hat g}\right)^{-\frac{2m}{n}} \\
& \quad - \frac{2m(m+n-1)\Lambda_m}{n(m+n-2)}\left(\int_M \hat v^{m-1}\dvol_{\hat g}\right)^{-\frac{2m+n}{n}}\hat v^{-1} \\
\frac{m+n-1}{4(m+n)}\widehat{R_\phi^{m+1}} & = \frac{(2m+n)(m+n+1)\Lambda_{m+1}}{n(m+n-1)}\left(\int_M \hat v^{m+1}\dvol_{\hat g}\right)^{\frac{2m}{n}} \\
& \quad - \frac{2(m+1)(m+n)\Lambda_{m+1}}{n(m+n-1)}\left(\int_M \hat v^{m+1}\dvol_{\hat g}\right)^{\frac{2m+n}{n}}\hat v^{-1}
\end{align*}
for $\Lambda_m:=\Lambda[g,v^m\dvol]$ and $\Lambda_{m+1}:=\Lambda[g,v^{m+1}\dvol]$.  The identity~\eqref{eqn:constants_rigidity_lapl} then follows from the general identity
\[ \widehat{R_\phi^{m+1}} = \widehat{R_\phi^m} - \frac{2}{m+1}\hat v^{-m-1}\hat\Delta\hat v^{m+1} , \]
the relationship~\eqref{eqn:constants_reln}, and the identity~\eqref{eqn:constants_rigidity_norms}.
\end{proof}

Using Theorem~\ref{thm:constants_reln}, we can now prove Theorem~\ref{thm:integral_characterization}.

\begin{proof}[Proof of Theorem~\ref{thm:integral_characterization}]

By the work of Aubin~\cite{Aubin1976}, Perelman~\cite{Perelman1}, Schoen~\cite{Schoen1984}, and Trudinger~\cite{Trudinger1968}, we need only consider the case $m\in\bN$.

Consider first the case $m=1$.  By the resolution of the Yamabe Problem, we know that there exists a smooth positive minimizer of $\Lambda[g]$.  If $\Lambda[g]\leq 0$, then Proposition~\ref{prop:monotone_in_m} implies that $\Lambda[g,v^1\dvol]\leq 0$.  If instead $\Lambda[g]>0$, then Theorem~\ref{thm:constants_reln} implies that
\begin{equation}
\label{eqn:Lambda1_est}
\Lambda[g,v^1\dvol] \leq \Lambda[\bR^n,dx^2,1^1\dvol]
\end{equation}
with equality if and only if $\Lambda[g]=\Lambda[\bR^n,dx^2]$ and equality holds in~\eqref{eqn:constants_reln} with $m=0$.  In particular, either $\Lambda[g,v^1\dvol]<\Lambda[\bR^n,dx^2,1^1\dvol]$ or a positive minimizer $w\in C^\infty(M)$ of $\Lambda[g]$ yields a positive minimizer $w^{\frac{n-1}{n-2}}$ of $\Lambda[g,v^1\dvol]$.  Furthermore, if equality holds in~\eqref{eqn:Lambda1_est}, then, by Theorem~\ref{thm:constants_reln}, $(M^n,g,v^1\dvol)$ is conformally equivalent to $(S^n,g_0,v_0^1\dvol)$ for $v_0\in C^\infty(M)$ a positive function such that
\begin{equation}
\label{eqn:first_spherical_harmonic}
-\Delta v_0 = n\left(v_0 - \fint_{S^n}v_0\dvol\right) .
\end{equation}
Hence $v_0^{-2}g_0$ is Einstein, yielding the desired result.

Consider next the case $m=2$.  From the previous paragraph, we know that a positive volume-normalized minimizer $w\in C^\infty(M)$ of $\Lambda[g,v^1\dvol]$ exists.  If $\Lambda[g,v^1\dvol]\leq 0$, then Proposition~\ref{prop:monotone_in_m} implies that $\Lambda[g,v^2\dvol]\leq 0$.  If instead $\Lambda[g,v^1\dvol]>0$, then Theorem~\ref{thm:constants_reln} implies that
\begin{equation}
\label{eqn:Lambda2_est}
\Lambda[g,v^2\dvol] \leq \Lambda[\bR^n,dx^2,1^2\dvol] .
\end{equation}
If equality holds, then $(M^n,g,v^1\dvol)$ is conformally equivalent to $(S^n,g_0,1^m\dvol)$.  Taking $w=1$ as the minimizer of $\Lambda[g_0,1^m\dvol]$ in~\eqref{eqn:constants_rigidity_norms} then yields a contradiction.  Hence the inequality~\eqref{eqn:Lambda2_est} is strict.

Finally, we may apply Proposition~\ref{prop:monotone_in_m} and Theorem~\ref{thm:constants_reln} inductively as above to deduce that for all $3\leq m\in\bN$,
\begin{equation}
\label{eqn:inductive_step}
\Lambda[g,v^m\dvol] \leq \frac{\Lambda[\bR^n,dx^2,1^m\dvol]}{\Lambda[\bR^n,dx^2,1^2\dvol]}\Lambda[g,v^2\dvol] < \Lambda[\bR^n,dx^2,1^m\dvol] . \qedhere
\end{equation}
\end{proof}
\section{Another interpretation of the weighted Yamabe constant}
\label{sec:product}

There are, to the best of the author's knowledge, now three different proofs of Theorem~\ref{thm:dd}.  The first proof is the original proof by Del Pino and Dolbeault~\cite{DelPinoDolbeault2002}, and takes the PDE approach.  The second proof is due to Cordero-Erausquin, Nazaret and Villani~\cite{CorderoErausquinNazaretVillani2004}, and takes an optimal transport approach to the problem.  The third proof is due to Bakry (cf.\ \cite{BakryGentilLedoux2012,CarlenFigalli2011}), and is based upon the observation that if $w\in C^\infty(\bR^n)$ is a minimizer~\eqref{eqn:dd_bubble} of the GN inequality~\eqref{eqn:dd}, then there is a constant $\tau$ such that
\[ f(x,y) := \left( w^{-\frac{2}{m+n-2}}(x) + \lv y\rv^2/\tau \right)^{-\frac{2m+n-2}{2}} \in C^\infty(\bR^{n+2m}) \]
is a minimizer for the sharp Sobolev inequality.

In this section we give the curved analogue of Bakry's observation.  Indeed, we directly relate the weighted Yamabe constant of $(M^n,g,v^m\dvol)$ to the Yamabe constant of $(M^n\times\bR^{2m},g\oplus v^2dy^2)$ for $2m\in\bN\cup\{0\}$; see Theorem~\ref{thm:weighted_yamabe_product_precise} below.

The main part of the computation is contained in the next two lemmas.

\begin{lem}
\label{lem:gamma_lemma}
Fix $k,l\geq 0$, $2m\in\bN$, and constants $a,\tau>0$.  Then
\begin{equation}
\label{eqn:gamma_power}
\int_{\bR^{2m}} \frac{\lv y\rv^{2l}}{\left(a+\lv y\rv^2/\tau\right)^{2m+k}} dy = \frac{\pi^m\Gamma(m+l)\Gamma(m+k-l)\tau^{m+l}}{\Gamma(m)\Gamma(2m+k)a^{m+k-l}}
\end{equation}
where $\Gamma(x)$ is Euler's gamma function.
\end{lem}

\begin{proof}

Both formulae follow immediately by writing the integrals in spherical coordinates and using the facts (see, for example, \cite{AbramowitzStegun})
\[ \Vol(S^{2m-1}) = \frac{2\pi^m}{\Gamma(m)} \qquad \text{and} \qquad \int_0^\infty\frac{t^{\alpha-1}}{(1+t)^{\alpha+\beta}} dt = \frac{\Gamma(\alpha)\Gamma(\beta)}{\Gamma(\alpha+\beta)} . \qedhere \]
\end{proof}

\begin{lem}
\label{lem:reln_of_conf_lapls}
Let $(M^n,g,1^m\dvol)$ be a compact smooth metric measure space with nonnegative weighted conformal Laplacian and $2m\in\bN\cup\{0\}$, and let $w\in C^\infty(M)$ be a positive function.  Given any $\tau>0$, define $f\in C^\infty(M\times\bR^{2m})$ by
\begin{equation}
\label{eqn:extension}
f(x,y) = \left( w^{-\frac{2}{m+n-2}}(x) + \lv y\rv^2/\tau \right)^{-\frac{2m+n-2}{2}} .
\end{equation}
Denote by $L$ the conformal Laplacian of $(M^n\times\bR^{2m},g\oplus dy^2)$ and by $L_\phi^m$ the weighted conformal Laplacian of $(M^n,g,1^m\dvol)$.  It holds that
\begin{equation}
\label{eqn:reln_of_qs} \mQ(f) \geq C\left(\frac{2m+n-2}{n(m+n-2)^2}\mQ(w)\right)^{\frac{n}{2m+n}}
\end{equation}
for
\[ C = (2m+n)(2m+n-2)\left(\frac{\pi^m\Gamma(m+n)}{\Gamma(2m+n)}\right)^{\frac{2}{2m+n}}\left(\frac{2m+n-2}{2(m+n-1)}\right)^{\frac{2m}{2m+n}} . \]
Moreover, equality holds in~\eqref{eqn:reln_of_qs} if and only if
\begin{equation}
\label{eqn:equality_product}
\tau = \frac{n\int w^{\frac{2(m+n-1)}{m+n-2}}}{2(L_\phi^mw,w)} .
\end{equation}
\end{lem}

\begin{proof}

In what follows, all integrals are computed with respect to the Riemannian measure on the specified (product) manifold.

First, as an immediate consequence of~\eqref{eqn:gamma_power}
\begin{equation}
\label{eqn:volume}
\int_{M^n\times\bR^{2m}} f^{\frac{2(2m+n)}{2m+n-2}} = \frac{\pi^m\tau^m\Gamma(m+n)}{\Gamma(2m+n)}\int_M w^{\frac{2(m+n)}{m+n-2}}
\end{equation}
and
\begin{align*}
\int_{M^n\times\bR^{2m}} Rf^2 = \frac{(2m+n-1)(2m+n-2)\pi^m\tau^m\Gamma(m+n)}{(m+n-1)(m+n-2)\Gamma(2m+n)}\int_M Rw^2 .
\end{align*}
Next, direct computation shows that
\[ \lv\nabla f\rv^2 = \left(\frac{2m+n-2}{2}\right)^2\left(w^{-\frac{2}{m+n-2}}+\lv y\rv^2/\tau\right)^{-\frac{2m+n}{2}}\left( \left|\nabla w^{-\frac{2}{m+n-2}}\right|^2 + 4\lv y\rv^2/\tau^2\right) . \]
Using~\eqref{eqn:gamma_power} again, we see that
\begin{align*}
\int_{M^n\times\bR^{2m}}\lv\nabla f\rv^2 & = \left(\frac{2m+n-2}{m+n-2}\right)^2\frac{\pi^m\tau^m\Gamma(m+n)}{\Gamma(2m+n)}\int_M\lv\nabla w\rv^2 \\
& \quad + \frac{m(2m+n-2)^2\pi^m\tau^{m-1}\Gamma(m+n)}{(m+n-1)\Gamma(2m+n)}\int_M w^{\frac{2(m+n-1)}{m+n-2}} .
\end{align*}
Combining these equations yields
\begin{equation}
\label{eqn:reln_of_conf_lapls_with_tau}
\left( Lf,f\right) = \left(\frac{2m+n-2}{m+n-2}\right)^2\frac{\pi^m\tau^{\frac{m(2m+n-2)}{2m+n}}\Gamma(m+n)}{\Gamma(2m+n)}\widetilde{\mW}(w,\tau)
\end{equation}
for
\[ \widetilde{\mW}(w,\tau) := \tau^{\frac{2m}{2m+n}}\left( L_\phi^mw,w\right) + \frac{m(m+n-2)^2}{m+n-1}\tau^{-\frac{n}{2m+n}}\int_M w^{\frac{2(m+n-1)}{m+n-2}} . \]
Since $(L_\phi^mw,w)\geq0$, \eqref{eqn:optimize} and~\eqref{eqn:optimizer} imply that
\[ \widetilde{\mW}(w,\tau) \geq \frac{2m+n}{2}\left(\frac{2(L_\phi^mw,w)}{n}\right)^{\frac{n}{2m+n}}\left(\frac{(m+n-2)^2}{m+n-1}\int_M w^{\frac{2(m+n-1)}{m+n-2}}\right)^{\frac{2m}{2m+n}} \]
with equality if and only if
\[ \tau = \frac{n\int w^{\frac{2(m+n-1)}{m+n-2}}}{2(L_\phi^mw,w)} . \]
Combining~\eqref{eqn:volume}, \eqref{eqn:reln_of_conf_lapls_with_tau}, and the above lower bound for $\widetilde{\mW}(w,\tau)$ yields the result.
\end{proof}

We now state and prove the main result of this section.

\begin{thm}
\label{thm:weighted_yamabe_product_precise}
Let $(M^n,g,v^m\dvol)$ be a compact smooth metric measure space with nonnegative weighted Yamabe constant $\Lambda[g,v^m\dvol]$ and $2m\in\bN\cup\{0\}$.  Set
\[ \tilde\Lambda_p := \inf\left\{ \mQ[M^n\times\bR^{2m},v^{-2}g\oplus dy^2,\dvol,0](f) \colon f \text{ is of the form~\eqref{eqn:extension}} \right\} . \]
Then
\begin{equation}
\label{eqn:weighted_yamabe_product_genl}
\frac{\Lambda[g,v^m\dvol]}{\Lambda[\bR^n,dx^2,1^m\dvol]} \geq \left(\frac{\tilde\Lambda_p}{\Lambda[\bR^{n+2m},dx^2]}\right)^{\frac{2m+n}{n}} .
\end{equation}
In particular,
\begin{equation}
\label{eqn:weighted_yamabe_product}
\frac{\Lambda[g,v^m\dvol]}{\Lambda[\bR^n,dx^2,1^m\dvol]} \geq \left(\frac{\Lambda[M^n\times\bR^{2m}, g\oplus v^2dy^2]}{\Lambda[\bR^{n+2m},dx^2]}\right)^{\frac{2m+n}{n}} ,
\end{equation}

Moreover, if $w$ is a minimizer of $\Lambda[g,v^m\dvol]$ and equality holds in~\eqref{eqn:weighted_yamabe_product_genl}, then the function $f$ defined by~\eqref{eqn:extension} for $\tau$ given by~\eqref{eqn:equality_product} is a minimizer of $\tilde\Lambda_p$.  
\end{thm}

\begin{proof}

Since both the weighted Yamabe constant and the metric $v^{-2}g$ are conformal invariants of $(M^n,g,v^m\dvol)$, it suffices to consider the case $v=1$.  Let $w\in C^\infty(M)$, define $\tau$ by~\eqref{eqn:equality_product}, and define $f\in C^\infty(M\times\bR^{2m})$ by~\eqref{eqn:extension}.  It follows from Lemma~\ref{lem:reln_of_conf_lapls} that
\[ \mQ(w) = C\left(\frac{\mQ(f)}{(2m+n)(2m+n-2)}\right)^{\frac{2m+n}{n}} \geq C\left(\frac{\tilde\Lambda_p}{(2m+n)(2m+n-2)}\right)^{\frac{2m+n}{n}} \]
for
\[ C = \frac{n(m+n-2)^2}{2m+n-2}\left(\frac{\Gamma(2m+n)}{\pi^m\Gamma(m+n)}\right)^{\frac{2}{n}}\left(\frac{2(m+n-1)}{2m+n-2}\right)^{\frac{2m}{n}} . \]
The result then follows from the formula~\eqref{eqn:dd_yamabe}.
\end{proof}

As a corollary, we prove Theorem~\ref{thm:nonexistence}.

\begin{proof}[Proof of Theorem~\ref{thm:nonexistence}]

Schoen proved~\cite{Schoen1989} that the Yamabe constant of $S^n\times\bR$ is equal to the Yamabe constant of $S^{n+1}$, and the minimizers are precisely the constant multiples of the function $(\cosh t)^{-\frac{n-1}{2}}$ for $t$ a choice of affine parameter of $\bR$.  Theorem~\ref{thm:blow_up} and Theorem~\ref{thm:weighted_yamabe_product_precise} thus imply that
\[ \Lambda[g_0,1^{1/2}\dvol]=\Lambda[\bR^n,dx^2,1^{1/2}\dvol] . \]
By Theorem~\ref{thm:weighted_yamabe_product_precise}, if there exists a positive minimizer $w\in C^\infty(S^n)$ of the weighted Yamabe constant of $(S^n,g_0,1^{1/2}\dvol)$, then the function
\[ f(x,t) = \left( w^{-\frac{4}{2n-3}}(x) + t^2/\tau \right)^{-\frac{n-1}{2}} \]
for $\tau$ as in~\eqref{eqn:equality_product} is a minimizer of the Yamabe constant of $S^n\times\bR$, contradicting Schoen's result.
\end{proof}
\section{On the uniqueness of minimizers}
\label{sec:uniqueness}

We conclude this article with a discussion of Conjecture~\ref{conj:weighted_obata}, both by giving an approach to proving it and explaining how, if true, it implies that minimizers of the weighted Yamabe constant of $(S^n,g_0,1^m\dvol)$ do not exist for any $m\in(0,1)$.

As stated in the introduction, Conjecture~\ref{conj:weighted_obata} is a generalization of results of Obata~\cite{Obata1971} and Perelman~\cite{Perelman1}.  Since their proofs are based on the same idea (cf.\ \cite{Case2010b}), we expect this idea to also establish Conjecture~\ref{conj:weighted_obata} for the total weighted scalar curvature functional~\eqref{eqn:total_weighted_scalar_curvature_functional}

To begin, observe that the assumption~\eqref{eqn:weighted_einstein} is equivalent to the assumption that $(M^n,g,v^m\dvol)$ is a critical point of the total weighted scalar curvature functional; i.e.\ \eqref{eqn:weighted_trfree_ric_tilde} vanishes for suitable choice of $\mu$.  This is a consequence of the following analogue of a result of D.-S.\ Kim and Y.\ H.\ Kim~\cite{Kim_Kim}.

\begin{lem}
\label{lem:kk_qe1}
Let $(M^n,g,v^m\dvol)$ be a connected smooth metric measure space such that $m>0$ and
\begin{equation}
\label{eqn:qe1_repeat}
\Ric_\phi^m - \frac{R_\phi^m}{2(m+n-1)}g = \frac{m+n-2}{2(m+n-1)}\lambda g
\end{equation}
for some constant $\lambda\in\bR$.  Then there is a constant $\mu\in\bR$ such that
\begin{equation}
\label{eqn:qe1_csc}
R_\phi^m + \frac{m(m+n-1)}{m+n-2}\mu v^{-1} = (m+n)\lambda .
\end{equation}
In particular, the tensor $\widetilde{E_\phi^m}$ defined by~\eqref{eqn:weighted_trfree_ric_tilde} vanishes.

Conversely, suppose that there is a constant $\mu\in\bR$ such that $\widetilde{E_\phi^m}=0$.  Then there is a constant $\lambda\in\bR$ such that~\eqref{eqn:qe1_csc}, and hence~\eqref{eqn:qe1_repeat}, holds.
\end{lem}

\begin{proof}

It follows from Proposition~\ref{prop:weighted_bianchi} (see also~\cite[(5.3a) and (5.6)]{Case2011t}) that
\begin{multline*}
\frac{m+n-2}{2(m+n-1)}v^{-1}d\left(vR_\phi^m\right) \\ = \delta_\phi\left(\Ric_\phi^m - \frac{R_\phi^m}{2(m+n-1)}g\right) - \frac{1}{m}\tr\left(\Ric_\phi^m - \frac{R_\phi^m}{2(m+n-1)}g\right)d\phi .
\end{multline*}
Hence if~\eqref{eqn:qe1_repeat} holds, then there is a constant $\mu$ such that~\eqref{eqn:qe1_csc} holds.  Combining~\eqref{eqn:qe1_repeat} and~\eqref{eqn:qe1_csc} then yields $E_\phi^m = \frac{m\mu v^{-1}}{2(m+n)}g$.

Conversely, if $E_\phi^m=\frac{m\mu v^{-1}}{2(m+n)}g$ for some constant $\mu$, then~\eqref{eqn:weighted_bianchi} implies that there exists a constant $\lambda$ such that~\eqref{eqn:qe1_csc}, and hence~\eqref{eqn:qe1_repeat}, holds.
\end{proof}

One of the main ingredients in Obata's proof of his theorem~\cite{Obata1971} is the observation that if $(M^n,g)$ is conformally Einstein, then there exists a positive function $u\in C^\infty(M)$ such that the tracefree part of the Hessian of $u$ is proportional to the tracefree part of the Ricci curvature of $g$.  The weighted analogue is as follows.

\begin{prop}
\label{prop:lie_deriv_prop}
Let $(M^n,g,v^m\dvol)$ be a compact smooth metric measure space with $m>0$, fix a constant $\mu\in\bR$, and let $E_\phi^m, \widetilde{E_\phi^m}$ be as in~\eqref{eqn:weighted_trfree_ric}.  Suppose additionally that there exists a positive function $u\in C^\infty(M)$ such that the smooth metric measure space
\begin{equation}
\label{eqn:ld_scms}
\left(M^n,\hat g,\hat v^m\dvol_{\hat g}\right) := \left( M^n, u^{-2}g, u^{-m-n}v^m\dvol_g \right)
\end{equation}
satisfies
\begin{equation}
\label{eqn:qe1_e}
\widehat{E_\phi^m} = \frac{m\hat\mu\hat v^{-1}}{2(m+n)}\hat g
\end{equation}
for some constant $\hat\mu\in\bR$.  Then the vector field $X=-(m+n-2)\nabla u$ is such that
\begin{equation}
\label{eqn:Lxphig}
\frac{1}{2}L_Xg + \frac{1}{m}d\phi(X)\,g = u\left(\widetilde{E_\phi^m} + \frac{1}{m}\tr\left(\widetilde{E_\phi^m}\right)g\right) + \frac{\mu u-\hat\mu}{2v}g .
\end{equation}
\end{prop}

\begin{remark}
In the case $m=0$, the correct conclusion is that~\eqref{eqn:Lxphig} holds after taking the projection onto the tracefree part of both sides of the equality.
\end{remark}

\begin{proof}

Using the formulae in~\cite[Proposition~4.4]{Case2010a}, we compute that
\begin{equation}
\label{eqn:e_transform}
\widehat{E_\phi^m} = E_\phi^m + (m+n-2)u^{-1}\left(\nabla^2 u - \frac{1}{m+n}\Delta_\phi u \,g\right) .
\end{equation}
Since $m>0$, we compute that
\[ \nabla^2u + \frac{1}{m}\lp\nabla u,\nabla\phi\rp\,g = \nabla^2 u - \frac{1}{m+n}\Delta_\phi u\,g + \frac{1}{m}\tr\left(\nabla^2u - \frac{1}{m+n}\Delta_\phi u\,g\right)g . \]
The result then follows from~\eqref{eqn:qe1_e} and~\eqref{eqn:e_transform}.
\end{proof}

\begin{cor}
\label{cor:obata_divergence}
Let $(M^n,g,v^m\dvol)$ be a smooth metric measure space, suppose that there are constants $\lambda,\mu\in\bR$ such that~\eqref{eqn:qe1_csc} holds, and let $\widetilde{E_\phi^m}$ be as in~\eqref{eqn:weighted_trfree_ric}.  Suppose additionally that there exists a positive function $u\in C^\infty(M)$ such that~\eqref{eqn:ld_scms} satisfies~\eqref{eqn:qe1_e} for some constant $\hat\mu\in\bR$.  Then
\begin{equation}
\label{eqn:obata_divergence}
\delta_\phi\left(\widetilde{E_\phi^m}(X)\right) = u\left[\left|\widetilde{E_\phi^m}\right|^2 + \frac{1}{m}\left(\tr\widetilde{E_\phi^m}\right)^2\right] + \frac{\mu u-\hat\mu}{2v}\tr\widetilde{E_\phi^m} .
\end{equation}
for $X=-\frac{m+n-2}{2}\nabla u$.
\end{cor}

\begin{proof}

Proposition~\ref{prop:weighted_bianchi} and the assumption~\eqref{eqn:qe1_csc} together imply that
\[ \delta_\phi\widetilde{E_\phi^m} = \frac{1}{m}\left(\tr\widetilde{E_\phi^m}\right)d\phi . \]
In particular, it follows that
\[ \delta_\phi\left(\widetilde{E_\phi^m}(X)\right) = \left\lp \widetilde{E_\phi^m}, \frac{1}{2}L_Xg + \frac{1}{m}d\phi(X)\,g \right\rp \]
for all vector fields $X$ on $M$.  Taking $X=-(m+n-2)\nabla u$ and applying Proposition~\ref{prop:lie_deriv_prop} thus yields the desired result.
\end{proof}

The difficulty in applying Corollary~\ref{cor:obata_divergence} to verify Conjecture~\ref{conj:weighted_obata} is that there does not seem to be any \emph{a priori} reason why (the integral of) the last summand of~\eqref{eqn:obata_divergence} should be nonnegative.  We overcome this in two special cases below.  However, we first need the following characterizations of smooth metric measure spaces admitting at least two solutions to~\eqref{eqn:qe1_repeat} in a given conformal class.

\begin{prop}
\label{prop:no_two_qe}
Let $(M^n,g,v^m\dvol_g)$ and $(M^n,\hat g,\hat v^m\dvol_{\hat g})$ be two pointwise conformally equivalent smooth metric measure spaces satisfying~\eqref{eqn:qe1_repeat} for constants $\lambda,\hat\lambda\in\bR$, respectively.  Write $\hat g=u^{-2}g$ and suppose that $u$ is nonconstant.
\begin{enumerate}
\item If $M$ is compact, then $m=1$ and $(M^n,g,v^m\dvol)$ is conformally equivalent to $(S^n,g,1^m\dvol)$.
\item If instead $(M^n,g)$ is complete, $\lambda=0$, and $\hat\lambda>0$, then $(M^n,g)$ is isometric to $(\bR^n,dx^2)$ and $u(x)=\frac{c}{2(m+n-1)}\lv x-x_0\rv^2+\frac{\hat\lambda}{2c}$ for $x_0\in\bR^n$ and $c>0$.
\end{enumerate}
\end{prop}

\begin{proof}

For convenience, denote by $P_\phi^m$ the tensor
\[ P_\phi^m := \Ric_\phi^m - \frac{R_\phi^m}{2(m+n-1)}g . \]
From~\cite[Proposition~4.4]{Case2010a} we compute that
\[ \widehat{P_\phi^m} = P_\phi^m + (m+n-2)u^{-1}\left(\nabla^2 u - \frac{1}{2}u^{-1}\lv\nabla u\rv^2\right)g . \]
Since $(M^n,g,v^m\dvol_g)$ and $(M^n,\hat g,\hat v^m\dvol_{\hat g})$ both satisfy~\eqref{eqn:qe1_repeat}, it follows that
\begin{equation}
\label{eqn:main_ode}
u^{-1}\nabla^2 u - \frac{1}{2}u^{-2}\lv\nabla u\rv^2 g = \frac{1}{2(m+n-1)}\left(\hat\lambda u^{-2}-\lambda\right)g .
\end{equation}
Suppose that $u$ is nonconstant.  Since $\nabla^2u$ is a multiple of the metric, we find that the set $U=\{p\in M\colon\lv\nabla u\rv(p)\not=0\}$ splits isometrically as a warped product $\left(I\times\Sigma^{n-1},dt^2\oplus\psi^2(t)h\right)$ for $I$ an open interval and $(\Sigma,h)$ some level set of $u$ in $U$.  Moreover, $u=u(t)$ and $\psi=ku^\prime(t)$ for some constant $k>0$ (cf.\ \cite{Brinkmann1925}).  Inserting this into~\eqref{eqn:main_ode} yields
\[ u^{-1}u^{\prime\prime} - \frac{1}{2}u^{-2}\left(u^\prime\right)^2 = \frac{1}{2(m+n-1)}\left(\hat\lambda u^{-2}-\lambda\right) . \]
This can be integrated, yielding a constant $c\in\bR$ such that
\begin{equation}
\label{eqn:no_two_qe_ode}
\left(u^\prime\right)^2 = -\frac{1}{m+n-1}\left(\lambda u^2 - 2cu + \hat\lambda\right) .
\end{equation}

First, suppose that $M$ is compact and $u$ is nonconstant.  It follows from~\eqref{eqn:no_two_qe_ode} that $\lambda,\hat\lambda>0$ and $c^2>\lambda\hat\lambda$.  Hence, after adding a constant to $t$ if necessary,
\[ u(t) = \frac{\sqrt{c^2-\lambda\hat\lambda}}{\lambda}\cos\left(\sqrt{\frac{\lambda}{m+n-1}}t\right) + \frac{c}{\lambda} . \]
It thus follows from the splitting of $U$ that $(M^n,g)$ is isometric to $(S^n,g_0)$.  A straightforward computation shows that $(S^n,g_0,v^m\dvol)$ satisfies~\eqref{eqn:qe1_repeat} with $\lambda>0$ and $v>0$ if and only if either $m=0$ or $m=1$ and $v(p)=a+b\cos\left(d(p,q)\right)$ for $q\in S^n$ and $a>b>0$.  Hence $v^{-2}g_0$ is again an Einstein metric on $S^n$, as desired.

Second, suppose that $(M^n,g)$ is complete, $u$ is nonconstant, $\lambda=0$, and $\hat\lambda>0$.  Since $\lambda=0$ and $u$ is nonconstant, \eqref{eqn:no_two_qe_ode} implies that $c>0$ and, after adding a constant to $t$ if necessary, $u(t)=\frac{c}{2(m+n-1)}t^2+\frac{\hat\lambda}{2c}$.  It thus follows from the splitting of $U$ and the completeness of $g$ that $(M^n,g)$ is isometric to $(\bR^n,g)$ and that $u$ is radially symmetric about some point $x_0\in\bR^n$, as desired.
\end{proof}

We now state and prove two special cases of Conjecture~\ref{conj:weighted_obata}.

\begin{prop}
\label{prop:weighted_obata_measure_critical}
Let $(M^n,g,v^m\dvol)$ be a compact smooth metric measure space such that~\eqref{eqn:qe1_csc} holds for constants $\lambda,\mu\in\bR$ and such that $(1,v)$ is a critical point of the map $(\xi,\nu)\mapsto\mQ[g,\nu^m\dvol](\xi)$.  Suppose additionally that there exists a positive function $u\in C^\infty(M)$ such that
\[ \left( M^n, \hat g, \hat v^m\dvol_{\hat g}\right) := \left( M^n, u^{-2}g, u^{-m-n}v^m\dvol_g \right) \]
satisfies~\eqref{eqn:qe1_repeat} for some constant $\hat\lambda\in\bR$.  Then $(M^n,g,v^m\dvol)$ satisfies~\eqref{eqn:qe1_repeat}.  Moreover, either $u$ is constant or $m\in\{0,1\}$ and $(M^n,g,v^m\dvol)$ is conformally equivalent to $(S^n,g_0,1^m\dvol)$.
\end{prop}

\begin{proof}

Proposition~\ref{prop:var_total_weighted_scalar_curvature} and the assumption that $(1,v)$ is a critical point of the map $(\xi,\nu)\mapsto\mQ[g,\nu^m\dvol](\xi)$ imply that $\widetilde{E_\phi^m}$ is tracefree.  Integrating~\eqref{eqn:obata_divergence} shows that $\widetilde{E_\phi^m}$ vanishes identically.  The result follows from Lemma~\ref{lem:kk_qe1} and Proposition~\ref{prop:no_two_qe}.
\end{proof}

\begin{prop}
\label{prop:weighted_obata_euclidean}
Fix $m\in[0,\infty]$ and consider $(\bR^n,dx^2,1^m\dvol)$ as a smooth metric measure space.  Let $w\in C^\infty(M)$ be a positive critical point of the weighted Yamabe quotient, and suppose additionally that the function
\[ \tilde w(x) := \lv x\rv^{2-m-n}w\left(\frac{x}{\lv x\rv^2}\right) \]
admits an extension to a positive element of $C^2(\bR^n)$.  Then there exist constants $a,b>0$ and a point $x_0\in\bR^n$ such that
\[ w(x) = \left(a + b\lv x-x_0\rv^2\right)^{-\frac{m+n-2}{2}} . \]
\end{prop}

\begin{proof}

By Proposition~\ref{prop:var_total_weighted_scalar_curvature}, the smooth metric measure space
\begin{equation}
\label{eqn:euclidean_conformal}
\left( S^n\setminus\{p\}, g, v^m\dvol_g \right) := \left( \bR^n, w^{\frac{4}{m+n-2}}dx^2, w^{\frac{2(m+n)}{m+n-2}}1^m\dvol_{dx^2} \right)
\end{equation}
is such that~\eqref{eqn:qe1_csc} holds for constants $\mu,\lambda\in\bR$, and moreover that the tensor $\widetilde{E_\phi^m}$ defined in terms of $\mu$ by~\eqref{eqn:weighted_trfree_ric} satisfies
\begin{equation}
\label{eqn:var_assumption_euclidean}
\int_{S^n\setminus\{p\}} \tr\left(\widetilde{E_\phi^m}\right) = 0 .
\end{equation}
Since $\hat\mu=0$ and $v=u$, Corollary~\ref{cor:obata_divergence} and~\eqref{eqn:var_assumption_euclidean} imply that
\begin{equation}
\label{eqn:integral_formula}
\int_{S^n\setminus\{p\}} \delta_\phi\left(\widetilde{E_\phi^m}(X)\right) = \int_{S^n\setminus\{p\}} \left[ \left|\widetilde{E_\phi^m}\right|^2 + \frac{1}{m}\left(\tr\widetilde{E_\phi^m}\right)^2 \right] u .
\end{equation}
We claim that the left hand side vanishes.  Indeed, suppose that the metric $g$ and the function $v$ defined by~\eqref{eqn:euclidean_conformal} can be extended to a $C^2$ metric and function, respectively, on $S^n$; note that $v(p)=0$.  In particular, $\lv\widetilde{E_\phi^m}\rv\leq Cv^{-2}$ for some constant $C>0$.  Since
\[ \int_{\{v>\varepsilon\}} \delta_\phi\left(\widetilde{E_\phi^m}(X)\right) = \int_{v^{-1}(\varepsilon)} \widetilde{E_\phi^m}(\nabla v,\nabla v)\lv\nabla v\rv^{-1}v^m\dvol \]
for any regular value $\varepsilon>0$ of $v$, taking $\varepsilon\to0$ yields the claim.

Finally, let us verify that the metric $g$ and the function $v$ defined by~\eqref{eqn:euclidean_conformal} admit $C^2$ extensions to $S^n$.  To that end, observe that $u:=w^{-\frac{2}{m+n-2}}$ has the property that $\tilde u(x):=\lv x\rv^{2}u(\frac{x}{\lv x\rv^2})$ can be extended to a positive $C^2$ function on $\bR^n$ if and only if $\tilde w$ can be extended to a positive $C^2$ function on $\bR^n$.  Fix $p\in S^n$ and denote by $r(\cdot)=d(p,\cdot)$ the distance from $p$ in $S^n$.  Stereographic projection from $p$ yields the isometry
\[ \left( S^n\setminus\{p\}, g_0, (1-\cos r)^m\dvol_{g_0}\right) = \left(\bR^n, u_0^{-2}dx^2, u_0^{-m-n}1^m\dvol_{dx^2}\right) . \]
for $u_0(x):=\frac{1+\lv x\rv^2}{2}$.  Using this observation, we see that $g$ admits a $C^2$ extension to $S^n$ if and only if $\tilde u$ admits a $C^2$ extension to $\bR^n$ (cf.\ \cite{Viaclovsky2000b}), and likewise $v$ admits a $C^2$ extension to $S^n$ if and only if $\tilde u$ admits a $C^2$ extension to $S^n$.
\end{proof}

We conclude with the following consequence of Conjecture~\ref{conj:weighted_obata} for the weighted Yamabe problem.

\begin{prop}
\label{prop:no_existence}
Suppose that Conjecture~\ref{conj:weighted_obata} is true.  Then given any $m\in(0,1)$, there does not exist a positive smooth minimizer of the weighted Yamabe quotient of $(S^n,g_0,1^m\dvol)$.
\end{prop}

\begin{proof}

Suppose to the contrary that a positive smooth minimizer $w$ exists.  If Conjecture~\ref{conj:weighted_obata} holds, then $w$ is constant.  Thus
\[ \Lambda := \Lambda[S^n,g_0,1^m\dvol] = \mQ(1) = \frac{n(n-1)(m+n-2)\pi}{m+n-1}\left(\frac{\Gamma(\frac{n}{2})}{\Gamma(n)}\right)^{\frac{2}{n}} . \]
In particular,
\begin{equation}
\label{eqn:LambdaSn_est}
\frac{\Lambda}{\Lambda_{m,n}} = \frac{(n-1)(2m+n-2)}{(m+n-1)(m+n-2)}\left(\frac{2m+n-2}{2(m+n-1)}\right)^{\frac{2m}{n}}\left(\frac{\Gamma(m+n)\Gamma(\frac{n}{2})}{\Gamma(\frac{2m+n}{2})\Gamma(n)}\right)^{\frac{2}{n}} .
\end{equation}
It follows from Lemma~\ref{lem:gamma} below that the right hand side is strictly greater than one when $m\in(0,1)$, contradicting Theorem~\ref{thm:blow_up}.
\end{proof}

\begin{lem}
\label{lem:gamma}
Define $F\colon[0,\infty]\times(2,\infty)\to\bR$ by
\[ F(m,n) = \frac{(n-1)(2m+n-2)}{(m+n-1)(m+n-2)}\left(\frac{2m+n-2}{2(m+n-1)}\right)^{\frac{2m}{n}}\left(\frac{\Gamma(m+n)\Gamma(\frac{n}{2})}{\Gamma(\frac{2m+n}{2})\Gamma(n)}\right)^{\frac{2}{n}} . \]
Then
\begin{enumerate}
\item $F(0,n)=1=F(1,n)$ for all $n$;
\item $F(m,n)>1$ for all $m\in(0,1)$ and all $n$; and
\item $F(m,n)<1$ for all $m>1$, $n$.
\end{enumerate}
\end{lem}

\begin{proof}

It is clear by direct computation that $F(0,n)=F(1,n)=1$ for all $n$.  We next show that the conclusion follows for all $n$ sufficiently large.  To that end, recall Stirling's approximation
\begin{equation}
\label{eqn:stirling}
\log\Gamma(x) = x\log x - x - \frac{1}{2}\log\frac{x}{2\pi} + \frac{1}{12x} + O(x^{-3})
\end{equation}
for $x$ large (cf.\ \cite{AbramowitzStegun}).  Fix $m\in[0,\infty)$.  Then~\eqref{eqn:stirling} implies that
\begin{align*}
\log\frac{\Gamma(m+n)\Gamma(\frac{n}{2})}{\Gamma(\frac{2m+n}{2})\Gamma(n)} & = m\log\frac{2(m+n-1)}{2m+n-2} + \frac{n}{2}\log\frac{(m+n-1)^2}{n(2m+n-2)} \\
& \quad - \frac{1}{2}\log\frac{m+n-1}{2m+n-2} + \frac{m-1}{4(m+n-1)(2m+n-2)} + O(n^{-3}) ,
\end{align*}
while the Maclaurin expansion of $\log(1+x)$ implies that
\[ \log\frac{(n-1)(2m+n-2)}{(m+n-1)(m+n-2)} = -\frac{m(m-1)}{(m+n-1)(m+n-2)} + O(n^{-4}) . \]
Combining these two expansions yields
\[ \log F(m,n) = -\frac{m(m-1)}{2n^3} + O(n^{-4}) , \]
thus establishing the claim for $n$ sufficiently large.

Now set $H(m,n)=F(m,n)^{n/2}$.  Clearly $H(m,n)$ is greater than (resp.\ less than) one if and only if $F(m,n)$ is greater than (resp.\ less than) one.  A straightforward computation shows that
\begin{equation}
\label{eqn:H_iterate_n}
\begin{split}
\log\frac{H(m,n+2)}{H(m,n)} & = \frac{n}{2}\log\left(\frac{(n+1)(2m+n)(m+n-1)(m+n-2)}{(m+n+1)(m+n)(n-1)(2m+n-2)}\right) \\
& \quad + m\log\left(\frac{(2m+n)(m+n-1)}{(m+n+1)(2m+n-2)}\right) .
\end{split}
\end{equation}

Using the estimate $\log(1+x)\leq x$ for all $x>0$, we see that
\[ \log\frac{H(m,n+2)}{H(m,n)} \leq \frac{m(m-1)(2m+n)}{(m+n)(m+n+1)(n-1)(2m+n-2)} . \]
In particular, if $m\in(0,1)$ we have that $H(m,n)\geq H(m,n+2)$.  Iterating this shows that $H(m,n)\geq H(m,n+2k)$ for $k$ sufficiently large, so that the previous paragraph yields $H(m,n)>1$ for all $m\in(0,1)$, $n>2$.

Using instead the estimate $\log(x)\geq\frac{x-1}{x}$ for all $x>0$, we compute from~\eqref{eqn:H_iterate_n} that
\[ \log\frac{H(m,n+2)}{H(m,n)} \geq \frac{m(m-1)(-2m+n+4)}{(2m+n)(m+n-1)(n+1)(m+n-2)} . \]
In particular, if $m\in(1,2]$ we have that $H(m,n)\leq H(m,n+2)$.  Iterating this and using the conclusion of the first paragraph yields $H(m,n)<1$ for all $m\in(1,2]$, $n>2$.

Finally, we compute that
\[ \frac{H(m+1,n)}{H(m,n)} = \left(\frac{(2m+n)(m+n-2)}{(2m+n-2)(m+n)}\right)^{\frac{n}{2}}\left(\frac{(2m+n)(m+n-1)}{(2m+n-2)(m+n)}\right)^m . \]
Taking the logarithm of both sides and using again the estimate $\log(1+x)\leq x$ yields $H(m+1,n)\leq H(m,n)$ for all $m>0$, $n>2$.  In particular, $H(m,n)\geq H(x(m),n)$ for all $m>1$, where $x(m)$ is the unique element of $(1,2]$ such that $m-x(m)\in\bZ$.  The result then follows from the previous paragraph.
\end{proof}

\begin{remark}

Lemma~\ref{lem:gamma} shows that $\Lambda[S^n,g_0,1^m\dvol]<\Lambda[\bR^n,dx^2,1^m\dvol]$ for all $m>1$, which is the reason we expect that the weighted Yamabe problem is solvable for all $m\geq1$.
\end{remark}

\bibliographystyle{abbrv}
\bibliography{../bib}

\newcommand{\noopsort}[1]{}
\begin{thebibliography}{10}

\bibitem{AbramowitzStegun}
M.~Abramowitz and I.~A. Stegun.
\newblock {\em Handbook of mathematical functions with formulas, graphs, and
  mathematical tables}, volume~55 of {\em National Bureau of Standards Applied
  Mathematics Series}.
\newblock For sale by the Superintendent of Documents, U.S. Government Printing
  Office, Washington, D.C., 1964.

\bibitem{Aubin1976}
T.~Aubin.
\newblock \'{E}quations diff\'erentielles non lin\'eaires et probl\`eme de
  {Y}amabe concernant la courbure scalaire.
\newblock {\em J. Math. Pures Appl. (9)}, 55(3):269--296, 1976.

\bibitem{Aubin1976s}
T.~Aubin.
\newblock Probl\`emes isop\'erim\'etriques et espaces de {S}obolev.
\newblock {\em J. Differential Geometry}, 11(4):573--598, 1976.

\bibitem{BakryGentilLedoux2012}
D.~Bakry, I.~Gentil, and M.~Ledoux.
\newblock {\em Analysis and geometry of {M}arkov diffusion operators}, volume
  348 of {\em Grundlehren der Mathematischen Wissenschaften [Fundamental
  Principles of Mathematical Sciences]}.
\newblock Springer, Cham, 2014.

\bibitem{Brinkmann1925}
H.~W. Brinkmann.
\newblock Einstein spaces which are mapped conformally on each other.
\newblock {\em Math. Ann.}, 94(1):119--145, 1925.

\bibitem{CarlenFigalli2011}
E.~A. Carlen and A.~Figalli.
\newblock Stability for a {GNS} inequality and the log-{HLS} inequality, with
  application to the critical mass {K}eller-{S}egel equation.
\newblock {\em Duke Math. J.}, 162(3):579--625, 2013.

\bibitem{CarlenLoss1993}
E.~A. Carlen and M.~Loss.
\newblock Sharp constant in {N}ash's inequality.
\newblock {\em Internat. Math. Res. Notices}, (7):213--215, 1993.

\bibitem{Case2010a}
J.~S. Case.
\newblock Smooth metric measure spaces and quasi-{E}instein metrics.
\newblock {\em Internat. J. Math.}, 23(10):1250110, 36 pp., 2012.

\bibitem{Case2011t}
J.~S. Case.
\newblock Smooth metric measure spaces, quasi-{E}instein metrics, and tractors.
\newblock {\em Cent. Eur. J. Math.}, 10(5):1733--1762, 2012.

\bibitem{Case2011gns}
J.~S. Case.
\newblock Conformal invariants measuring the best constants for
  {G}agliardo-{N}irenberg-{S}obolev inequalities.
\newblock {\em Calc. Var. Partial Differential Equations}, 48(3-4):507--526,
  2013.

\bibitem{Case2010b}
J.~S. Case.
\newblock The {E}nergy of a {S}mooth {M}etric {M}easure {S}pace and
  {A}pplications.
\newblock {\em J. Geom. Anal.}, 25(1):616--667, 2015.

\bibitem{CorderoErausquinNazaretVillani2004}
D.~Cordero-Erausquin, B.~Nazaret, and C.~Villani.
\newblock A mass-transportation approach to sharp {S}obolev and
  {G}agliardo-{N}irenberg inequalities.
\newblock {\em Adv. Math.}, 182(2):307--332, 2004.

\bibitem{DelPinoDolbeault2002}
M.~Del~Pino and J.~Dolbeault.
\newblock Best constants for {G}agliardo-{N}irenberg inequalities and
  applications to nonlinear diffusions.
\newblock {\em J. Math. Pures Appl. (9)}, 81(9):847--875, 2002.

\bibitem{GromovLawson1983}
M.~Gromov and H.~B. Lawson, Jr.
\newblock Positive scalar curvature and the {D}irac operator on complete
  {R}iemannian manifolds.
\newblock {\em Inst. Hautes \'Etudes Sci. Publ. Math.}, (58):83--196 (1984),
  1983.

\bibitem{Hebey1999}
E.~Hebey.
\newblock {\em Nonlinear analysis on manifolds: {S}obolev spaces and
  inequalities}, volume~5 of {\em Courant Lecture Notes in Mathematics}.
\newblock New York University Courant Institute of Mathematical Sciences, New
  York, 1999.

\bibitem{Kim_Kim}
D.-S. Kim and Y.~H. Kim.
\newblock Compact {E}instein warped product spaces with nonpositive scalar
  curvature.
\newblock {\em Proc. Amer. Math. Soc.}, 131(8):2573--2576 (electronic), 2003.

\bibitem{LeeParker1987}
J.~M. Lee and T.~H. Parker.
\newblock The {Y}amabe problem.
\newblock {\em Bull. Amer. Math. Soc. (N.S.)}, 17(1):37--91, 1987.

\bibitem{Obata1971}
M.~Obata.
\newblock The conjectures on conformal transformations of {R}iemannian
  manifolds.
\newblock {\em J. Differential Geometry}, 6:247--258, 1971/72.

\bibitem{Perelman1}
G.~Perelman.
\newblock The entropy formula for the {R}icci flow and its geometric
  applications.
\newblock arXiv:0211159, \noopsort{2097}preprint.

\bibitem{Schoen1984}
R.~Schoen.
\newblock Conformal deformation of a {R}iemannian metric to constant scalar
  curvature.
\newblock {\em J. Differential Geom.}, 20(2):479--495, 1984.

\bibitem{Schoen1989}
R.~M. Schoen.
\newblock Variational theory for the total scalar curvature functional for
  {R}iemannian metrics and related topics.
\newblock In {\em Topics in calculus of variations ({M}ontecatini {T}erme,
  1987)}, volume 1365 of {\em Lecture Notes in Math.}, pages 120--154.
  Springer, Berlin, 1989.

\bibitem{Talenti1976}
G.~Talenti.
\newblock Best constant in {S}obolev inequality.
\newblock {\em Ann. Mat. Pura Appl. (4)}, 110:353--372, 1976.

\bibitem{Trudinger1968}
N.~S. Trudinger.
\newblock Remarks concerning the conformal deformation of {R}iemannian
  structures on compact manifolds.
\newblock {\em Ann. Scuola Norm. Sup. Pisa (3)}, 22:265--274, 1968.

\bibitem{Viaclovsky2000b}
J.~A. Viaclovsky.
\newblock Conformally invariant {M}onge-{A}mp\`ere equations: global solutions.
\newblock {\em Trans. Amer. Math. Soc.}, 352(9):4371--4379, 2000.

\bibitem{Yamabe1960}
H.~Yamabe.
\newblock On a deformation of {R}iemannian structures on compact manifolds.
\newblock {\em Osaka Math. J.}, 12:21--37, 1960.

\end{thebibliography}
\end{document}